\documentclass[a4paper,16pt]{article}
\usepackage[utf8]{inputenc}
\usepackage{amssymb,amsmath,mathrsfs, MnSymbol}
\usepackage{amsthm}
\usepackage[colorlinks,
            linkcolor=blue,
            anchorcolor=red,
            citecolor=green
            ]{hyperref}
\numberwithin{equation}{section}
\usepackage{graphicx}
\usepackage{amsmath}
\newtheorem{theorem}{Theorem}[section]

\theoremstyle{plain}
\newtheorem{thm}[theorem]{Theorem } 

\newtheorem{prop}[theorem]{Proposition }

\newtheorem{lem}[theorem]{Lemma }

\newtheorem{rem}[theorem] {Remark}

\usepackage{tikz}
\usetikzlibrary{matrix}

\begin{document}

\title{Matrix solutions of the cubic Szeg\H{o} equation on the real line}
\author{Ruoci Sun\footnote{School of Mathematics, Georgia Institute of Technology, Atlanta, USA. Email: ruoci.sun.16@normalesup.org}}
\date{}

\maketitle

\noindent $\mathbf{Abstract}$ \quad This paper is dedicated to studying matrix solutions of the cubic Szeg\H{o} equation on the line in Pocovnicu \cite{poAPDE2011, poDCDS2011} and  G\'erard--Pushnitski \cite{GerPush2023},  leading to the following matrix Szeg\H{o} equation on $\mathbb{R}$,
\begin{equation*}
 i \partial_t U = \Pi_{\geq 0} \left(U  U ^* U \right), \quad  \widehat{\left(\Pi_{\geq 0}  U\right)}(\xi )= \mathbf{1}_{\xi \geq 0}\hat{U}(\xi )\in \mathbb{C}^{M \times N} .
 \end{equation*}Inspired from the space-periodic case in Sun \cite{SUNMSzego}, we establish its Lax pair structure via double Hankel operators and Toeplitz operators. Then the explicit formula in \cite{GerPush2023} can be extended to two equivalent formulas  in the matrix equation case, which both express every solution explicitly in terms of its initial datum and the time variable.\\

\noindent $\mathbf{Keywords}$ \quad Szeg\H{o} operator, Lax pair, explicit formula, Hankel operators, Toeplitz operators, Lax--Beurling shift semigroup.\\
\tableofcontents

\newpage

\section{Introduction}
\noindent For any positive integers $M,N \in \mathbb{N}_+$, the cubic $M \times N$ matrix Szeg\H{o} equation on the real line reads as
\begin{equation}\label{MSzego}
 i \partial_t U = \Pi_{\geq 0} \left(U  U ^* U \right), \quad U=U(t,x) \in \mathbb{C}^{M \times N}, \quad (t,x)\in \mathbb{R} \times \mathbb{R},
 \end{equation}where $\Pi_{\geq 0} = \Pi_{\geq 0}^{\mathbb{R}}: L^2(\mathbb{R}; \mathbb{C}^{M \times N}) \to L^2(\mathbb{R}; \mathbb{C}^{M \times N})$ denotes Szeg\H{o} projector on  $L^2(\mathbb{R}; \mathbb{C}^{M \times N})$, which is the Fourier multiplier of symbol Heaviside step function, i.e.  $\Pi_{\geq 0} = \mathbf{1}_{[0, +\infty)}(\mathrm{D})$ with $\mathrm{D}=-i \partial_x$, i.e.
\begin{equation}\label{MSzegoopR}
\widehat{\left(\Pi_{\geq 0}  U\right)}(\xi_1)= \hat{U}(\xi_1)\in \mathbb{C}^{M \times N}, \quad \widehat{\left(\Pi_{\geq 0}  U\right)}(\xi_2)=0_{M \times N}, \quad \forall \xi_1>0> \xi_2 .
\end{equation}for any $U \in L^2(\mathbb{R}; \mathbb{C}^{M \times N})$.   
\subsection{Motivation}
\noindent The motivation to introduce equation \eqref{MSzego} is based on the following two facts. On the one hand, the cubic scalar Szeg\H{o} equation on the torus $\mathbb{T}:=\mathbb{R}\slash 2 \pi \mathbb{Z}$,
\begin{equation}\label{sSzegoT}
i\partial_t u = \Pi^{\mathbb{T}}_{\geq 0}(|u|^2 u), \quad u=u(t,x) \in \mathbb{C}, \quad (t,x) \in \mathbb{R} \times \mathbb{T}, \quad \Pi_{\geq 0}^{\mathbb{T}}: \sum_{n\in \mathbb{Z}}a_n e^{inx} \mapsto \sum_{n \geq 0}a_n e^{inx}, 
\end{equation}which is a toy model of totally nondispersive Hamiltonian equation, is introduced in G\'erard--Grellier \cite{GGANNENS, GGinvent, GGAPDE,  GGTurb2015, GGTAMS, GerardGrellierBook} and  G\'erard--Pushnitski \cite{GerPushInv2022} in order to discover new tools to study the problem of global wellposedness and other qualitative properties of smooth solutions of the nonlinear Schr\"odinger-type equation which is lack of dispersion. Thanks to its two-Lax-pair structure, P. G\'erard and S. Grellier have constructed action--angle coordinates on the finite rank manifolds and the explicit formula for general solutions, leading to its complete integrability. The explicit formula in G\'erard--Grellier \cite{GGTAMS} allows to extend the flow map of \eqref{sSzegoT} to the low  regularity phase space $L^2_+(\mathbb{T}; \mathbb{C}):=\Pi_{\geq 0}^{\mathbb{T}}(L^2 (\mathbb{T}; \mathbb{C}))$ and this extension is sharp, according to the pioneering work G\'erard--Pushnitski \cite{GerPushInv2022}. The classification of traveling waves, the nonlinear Fourier transform and turbulent solutions of \eqref{sSzegoT} have been established in G\'erard--Grellier \cite{GGANNENS,    GGTurb2015,  GerardGrellierBook}. O.  Pocovnicu has introduced the scalar version of \eqref{MSzego} in \cite{poAPDE2011, poDCDS2011},
\begin{equation}\label{sSzegoR}
i\partial_t u = \Pi^{\mathbb{R}}_{\geq 0}(|u|^2 u), \quad u=u(t,x) \in \mathbb{C}, \quad (t,x) \in \mathbb{R} \times \mathbb{R}, \quad \Pi_{\geq 0}^{\mathbb{R}} = \mathbf{1}_{(0, +\infty)}(\mathrm{D}), 
\end{equation}in order to compare the space-periodic solutions with the space non-periodic solutions of the cubic scalar Szeg\H{o} equation and to study their similarities and differences. The cubic scalar Szeg\H{o} equation on $\mathbb{R}$ also enjoys a Lax pair structure, which allows to establish generalized action--angle coordinates, the soliton resolution, the classification of traveling waves and to construct the turbulent solutions, according to \cite{poAPDE2011, poDCDS2011}. The explicit formula for general solutions of \eqref{sSzegoR} is discovered by G\'erard--Pushnitski \cite{GerPush2023} and it allows to extend its flow map to the low regularity phase space $L^2_+(\mathbb{R}; \mathbb{C}):=\Pi_{\geq 0}^{\mathbb{R}}(L^2 (\mathbb{R}; \mathbb{C}))$. However, the sharpness of this extension still remains as an open problem. There are also many other significant differences on the dynamics between the space-periodic solutions  \eqref{sSzegoT} and the space non-periodic solutions \eqref{sSzegoR}. For instance, every rational solution to \eqref{sSzegoT} is quasi-periodic; nevertheless, the high regularity Sobolev norm of rational solutions \eqref{sSzegoR} may tend to infinity for some  dense subset of rational initial data. The space non-periodic explicit formula  in  G\'erard--Pushnitski \cite{GerPush2023} shows a completely different nature from the space-periodic explicit formula in G\'erard--Grellier \cite{GGTAMS}. One cannot be deduced directly from another.\\

\noindent On the other hand, the matrix extension of the space-periodic cubic Szeg\H{o} equation \eqref{sSzegoT} in Sun \cite{SUNMSzego} displays a significant difference from the matrix generalizations of other integrable PDEs  including the Korteweg--de Vries (KdV) equation in  Lax \cite{LaxPairCPAMKdV}, the cubic Schr\"odinger system (NLS) in  Zakharov--Shabat \cite{ZS1972}, the spin Benjamin--Ono equation (sBO) in Berntson--Langmann--Lenells \cite{sBOBLL2022} and G\'erard \cite{sBOLaxP2022}, the Calogero--Moser--Sutherland derivative Schr\"odinger equation (CMSdNLS) in  G\'erard--Lenzmann \cite{Gerard-Lenzmann2022} and Badreddine \cite{Badreddine2023},  and the intertwined CMSdNLS system (iCMSdNLS) of two variables in Sun \cite{SUNInterCMSDNLS}. When generalizing to matrix solutions in the case of the KdV, NLS, sBO, CMSdNLS, iCMSdNLS equations, if the scalar multiplication is replaced by the right multiplication of matrices, then the Lax pair of the original scalar equation becomes the Lax pair of the corresponding matrix equation. However, it doesn't work for the cubic Szeg\H{o} equation on the torus according to Sun \cite{SUNMSzego}. The scalar Hankel operator 
\begin{equation}\label{Huscalar}
H_u : h \in L^2_+(\mathbb{M}; \mathbb{C}) \mapsto \Pi_{\geq 0}^{\mathbb{M}}(u \overline{h})\in L^2_+(\mathbb{M}; \mathbb{C}), \quad u \in \Pi_{\geq 0}^{\mathbb{M}} ( H^{\frac{1}{2}}(\mathbb{M}; \mathbb{C})  ), \quad \mathbb{M} \in \{\mathbb{T}, \mathbb{R}\},
\end{equation}has two matrix generalizations, the left and right matrix Hankel operators, defined by
\begin{equation}\label{rlHankelIntro}
\begin{split}
& \mathbf{H}^{(\mathbf{r})}_{U} : F \in L^2_+(\mathbb{M}; \mathbb{C}^{d \times N}) \mapsto \mathbf{H}^{(\mathbf{r})}_{U}(F) = \Pi_{\geq 0}(U F^*) \in L^2_+(\mathbb{M}; \mathbb{C}^{M \times d});\\
 & \mathbf{H}^{(\mathbf{l})}_{U} : G \in L^2_+(\mathbb{M}; \mathbb{C}^{M \times d}) \mapsto \mathbf{H}^{(\mathbf{l})}_{U}(G) = \Pi_{\geq 0}(G^* U) \in L^2_+(\mathbb{M}; \mathbb{C}^{d \times N}),
\end {split}
\end{equation}for any $M,N,d \in \mathbb{N}_+$ and $U \in H^{\frac{1}{2}}_+(\mathbb{M}; \mathbb{C}^{M\times N}) = \Pi_{\geq 0}^{\mathbb{M}}( H^{\frac{1}{2}} (\mathbb{M}; \mathbb{C}^{M\times N}) )$. When $M \ne N$, due to the rules of matrix addition and multiplication, neither   $\mathbf{H}^{(\mathbf{r})}_{U}$ nor $\mathbf{H}^{(\mathbf{l})}_{U}$ can be chosen as the Lax operator of the matrix Szeg\H{o} equation \eqref{MSzego},  while the scalar Hankel operator $H_u$ in \eqref{Huscalar}  is the Lax operator for the scalar Szeg\H{o} equation \eqref{sSzegoR}. We refer to G\'erard--Grellier \cite{GGANNENS, GGinvent, GGAPDE,  GGTurb2015, GGTAMS, GerardGrellierBook}, Pocovnicu \cite{poAPDE2011, poDCDS2011} and G\'erard--Pushnitski \cite{GerPushInv2022, GerPush2023}   for details. According to Sun \cite{SUNMSzego}, the double matrix Hankel operators $\mathbf{H}^{(\mathbf{r})}_{U} \mathbf{H}^{(\mathbf{l})}_{U}$ and  $\mathbf{H}^{(\mathbf{l})}_{U} \mathbf{H}^{(\mathbf{r})}_{U}$ remain to provide the Lax pair structure for the matrix Szeg\H{o} equation \eqref{MSzego} on $\mathbb{T}$, which extends the scalar explicit formula in G\'erard--Grellier \cite{GGTAMS} to every $H^{\frac{1}{2}}_+(\mathbb{T}; \mathbb{C}^{M\times N})$-solution. 
 As a consequence, the matrix extension of the cubic Szeg\H{o} equation \eqref{sSzegoT} allows to discover its interior structure of integrable system and the general Lax pair and explicit formula which hold for every matrix solution. \\
 
\noindent Inspired from G\'erard--Pushnitski \cite{GerPush2023} and Sun \cite{SUNMSzego}, we want to extend the explicit formula of scalar space non-periodic solutions of \eqref{sSzegoR} to every matrix solution, i.e. the solution to the matrix Szeg\H{o} equation \eqref{MSzego} on the real line. Let $H^s_+(\mathbb{R}; \mathbb{C}^{M\times N}) = \Pi_{\geq 0}^{\mathbb{R}}( H^{s} (\mathbb{R}; \mathbb{C}^{M\times N}) )$ denotes the filtered Sobolev space, $\forall s \geq 0$. Before stating the main result, we show that \eqref{MSzego} is globally wellposed in every high   regularity filtered Sobolev space. 
\begin{prop}\label{GWPH0.5R}
Given $U_0 \in H^{\frac{1}{2}}_+(\mathbb{R};  \mathbb{C}^{M\times N})$, there exists a unique function $U \in C(\mathbb{R}; H^{\frac{1}{2}}_+(\mathbb{R};  \mathbb{C}^{M\times N}))$ solving the cubic matrix Szeg\H{o} equation \eqref{MSzego} such that $U(0)=U_0$. For each $T>0$, the flow map $\Phi : U_0 \in H^{\frac{1}{2}}_+(\mathbb{R};  \mathbb{C}^{M\times N}) \mapsto U\in C([-T, T]; H^{\frac{1}{2}}_+(\mathbb{R};  \mathbb{C}^{M\times N}))$ is continuous. Moreover, if $U_0 \in H^s_+(\mathbb{R};  \mathbb{C}^{M\times N})$ for some $s > \frac{1}{2}$, then $U \in C^{\infty}(\mathbb{R}; H^s_+(\mathbb{R};  \mathbb{C}^{M\times N}))$.
\end{prop}
\noindent Given $d, M,N\in\mathbb{N}_+$, the Toeplitz operators of symbol $V \in L^{\infty}(\mathbb{R}; \mathbb{C}^{M \times N})$ are given by
\begin{equation}\label{rlToepIntro}
\small
 \mathbf{T}^{(\mathbf{r})}_{V}(G) = \Pi_{\geq 0}(V G) \in L^2_+(\mathbb{R}; \mathbb{C}^{M \times d}), \quad
 \mathbf{T}^{(\mathbf{l})}_{V}(F) = \Pi_{\geq 0}(FV) \in L^2_+(\mathbb{R}; \mathbb{C}^{d \times N}), 
\end{equation}for any $\forall (F,G) \in L^2_+(\mathbb{R}; \mathbb{C}^{d \times M}) \times L^2_+(\mathbb{R}; \mathbb{C}^{N \times d})$. The Lax pair structure of the matrix Szeg\H{o} equation \eqref{MSzego} is recalled in the next proposition.   
\begin{prop}[Sun \cite{SUNMSzego}]\label{LaxPairThmR}
If $U \in C^{\infty}(\mathbb{R}; H^s_+(\mathbb{R};  \mathbb{C}^{M\times N}))$ solves the matrix Szeg\H{o} equation \eqref{MSzego} for some $s>\tfrac{1}{2}$, for some positive integers $M,N \in \mathbb{N}_+$, then for every $d \in \mathbb{N}_+$, the time-dependent operators $\mathbf{H}^{(\mathbf{r})}_{U}\mathbf{H}^{(\mathbf{l})}_{U} \in C^{\infty}(\mathbb{R}; \mathcal{B}(L^2_+(\mathbb{R};  \mathbb{C}^{M\times d})))$ and  $\mathbf{H}^{(\mathbf{l})}_{U}\mathbf{H}^{(\mathbf{r})}_{U} \in C^{\infty}(\mathbb{R}; \mathcal{B}(L^2_+(\mathbb{R};  \mathbb{C}^{d\times N})))$ satisfy the following Heisenberg--Lax equations :
\begin{equation}\label{TwoHeiLaxMSzegoR}
\small
\begin{split}
& \tfrac{\mathrm{d}}{\mathrm{d}t}(\mathbf{H}^{(\mathbf{r})}_{U(t)}\mathbf{H}^{(\mathbf{l})}_{U(t)}) = i[\mathbf{H}^{(\mathbf{r})}_{U(t)}\mathbf{H}^{(\mathbf{l})}_{U(t)}, \; \mathbf{T}^{(\mathbf{r})}_{U(t) U(t)^*}]; \quad \tfrac{\mathrm{d}}{\mathrm{d}t}(\mathbf{H}^{(\mathbf{l})}_{U(t)}\mathbf{H}^{(\mathbf{r})}_{U(t)}) = i[\mathbf{H}^{(\mathbf{l})}_{U(t)}\mathbf{H}^{(\mathbf{r})}_{U(t)},  \; \mathbf{T}^{(\mathbf{l})}_{U(t)^* U(t)}].
\end{split}
\end{equation}
\end{prop}
\begin{rem}
Since the right Hankel operator  $\mathbf{H}^{(\mathbf{r})}_{U}$ coincides with the left Hankel operator $\mathbf{H}^{(\mathbf{l})}_{U}$ in the scalar case, when  $M=N=1$,   the single Hankel operator $H_u$ becomes a Lax operator of the cubic scalar Szeg\H{o} equation \eqref{sSzegoR} on the real line.
\end{rem}
\begin{rem}
In Sun \cite{SUNMSzego}, both the double Hankel operators and the double shift-Hankel operators are Lax operators of the matrix Szeg\H{o} equation on the torus. However, the double shift-Hankel operators are not Lax operators for the matrix Szeg\H{o} equation  \eqref{MSzego} on the real line. The construction of the explicit expression of its solutions relies only on the Lax operators in proposition $\ref{LaxPairThmR}$.
\end{rem}
 
\subsection{The main result}
\noindent We define $L^2_+(\mathbb{R}; \mathbb{C}^{M \times N})=\Pi_{\geq 0}(L^2 (\mathbb{R}; \mathbb{C}^{M \times N})) = H^0_+(\mathbb{R}; \mathbb{C}^{M \times N})$. The Poisson integral of any function $U \in L^2_+(\mathbb{R}; \mathbb{C}^{M \times N})$ is a holomorphic function on the upper half plane $\mathbb{C}_+:=\{z \in \mathbb{C}: \mathrm{Im}z >0\}$, given by  
\begin{equation}\label{PoiIntforR}
\underline{U}(z) = \mathscr{P}[U](z) : =  \int_{\mathbb{R}} \mathfrak{P}_{y}(x-t)U(t) \mathrm{d}t =\tfrac{1}{2\pi} \int_{0}^{+\infty} e^{iz \xi}\hat{U}(\xi)\mathrm{d}\xi, \quad \forall z=x+yi \in \mathbb{C}_+,
\end{equation}where $\mathfrak{P}_{y}(x)=\frac{y}{\pi(x^2 +y^2)}$, denotes the Poisson kernel on $\mathbb{C}_+$. The original function $U \in L^2_+(\mathbb{R}; \mathbb{C}^{M \times N})$ can be considered as the $L^2$-limit of its Poisson integral when $\mathrm{Im}z \to 0^+$, i.e.
\begin{equation}\label{CVPoiIntforR}
\lim_{y \to 0^+}\|\tau_{-yi}\underline{U}|_{\mathbb{R}} - U \|_{L^2_+(\mathbb{R}; \mathbb{C}^{M \times N})} =0.
\end{equation}where $\tau_{-yi}\underline{U}|_{\mathbb{R}}: x \in \mathbb{R} \mapsto \underline{U}(x +yi) \in \mathbb{C}^{M \times N} \in L^2_+(\mathbb{R}; \mathbb{C}^{M \times N})$, $\forall y >0$. Moreover, we have 
\begin{equation} 
\sup_{y>0}\|\tau_{-yi}\underline{U}|_{\mathbb{R}}\|_{L^2_+(\mathbb{R}; \mathbb{C}^{M \times N})} \leq \|U\|_{L^2_+(\mathbb{R}; \mathbb{C}^{M \times N})}.
\end{equation}Then $L^2_+(\mathbb{R}; \mathbb{C}^{M \times N})$ is identified as the following 
Hardy space via the $\mathbb{C}$-Hilbert isomorphism $U \mapsto \underline{U}$,
\begin{equation}\label{IdenHardyR}
\mathbb{H}^2(\mathbb{C}_+; \mathbb{C}^{M \times N}):= \{U \in \mathrm{Hol}(\mathbb{C}_+; \mathbb{C}^{M \times N}): \sup_{y>0}\int_{\mathbb{R}} \mathrm{tr}\left(\underline{U}(x+yi) (\underline{U}(x+yi))^*\right)\mathrm{d}x<+\infty\}. 
\end{equation}Since every $L^2_+(\mathbb{R}; \mathbb{C}^{M \times N})$-function is identified as its Poisson integral via the holomorphic Fourier transform \eqref{PoiIntforR}, the goal of this paper is to express the Poisson integral of every solution of the matrix Szeg\H{o} equation \eqref{MSzego} in terms of its initial datum and the time variable. The Lax--Beurling shift semigroup $(\mathtt{S}(\eta))_{\eta \geq 0}$ of isometries on $L^2_+(\mathbb{R}; \mathbb{C}^{M \times N})$ and its adjoint semigroup $(\mathtt{S}(\eta)^*)_{\eta \geq 0}$ are defined as 
\begin{equation}\label{LaxBeurlingSG}
\mathtt{S}(\eta)U = \mathbf{e}_{\eta}U, \quad \mathtt{S}(\eta)^*U = \Pi_{\geq 0} \left(\mathbf{e}_{ \eta}^{-1}U \right), \quad \mathbf{e}_{\eta}(x) = e^{i\eta x}, \quad \forall x \in \mathbb{R},\quad \forall \eta \geq 0,
\end{equation}for every $U \in L^2_+(\mathbb{R}; \mathbb{C}^{M \times N})$. Let $-i\mathbf{G}$ denote the infinitesimal generator of the contraction semigroup $(\mathtt{S}(\eta)^*)_{\eta \geq 0}$, i.e. $ \mathbf{G} (F):=i\frac{\mathrm{d}}{\mathrm{d}\eta}\big|_{\eta=0^+}\mathtt{S}(\eta)^*(F) \in L^2_+(\mathbb{R}; \mathbb{C}^{M \times N})$, $\forall F \in \mathrm{Dom}(\mathbf{G})^{M\times N}$, where
\begin{equation}\label{DomG}
\mathrm{Dom}(\mathbf{G})^{M\times N} := \{F \in L^2_+(\mathbb{R}; \mathbb{C}^{M \times N}) : \hat{F}|_{\mathbb{R}_+^*} \in H^1(\mathbb{R}_+^*; \mathbb{C}^{M \times N})\},
\end{equation}with $\mathbb{R}_+^* = (0, +\infty)$. For any $F \in L^2_+(\mathbb{R}; \mathbb{C}^{M \times N})$ such that $\hat{F} : \mathbb{R} \to \mathbb{C}^{M \times N}$ is right continuous at $0$, set
\begin{equation}\label{IntOpforR}
\mathscr{I} (F):= \hat{F}(0^+) \in \mathbb{C}^{M \times N}.
\end{equation}For any $F \in \mathrm{Dom}(\mathbf{G})^{M\times N}$, we have the following expression for $\mathbf{G}(F) \in L^2_+(\mathbb{R}; \mathbb{C}^{M \times N})$,
\begin{equation}\label{defGFIntro}
\mathbf{G}(F)(x)= xF(x) - \tfrac{i }{2\pi}\hat{F}(0^+), \quad \forall x \in \mathbb{R}. 
\end{equation}For any $U \in H^{\frac{1}{2}}_+(\mathbb{R}; \mathbb{C}^{M \times N})$, the following projection operators $\mathfrak{m}_U^{(\mathbf{rl})}:  L^2(\mathbb{R}; \mathbb{C}^{M \times d}) \to  H^{\frac{1}{2}}_+(\mathbb{R}; \mathbb{C}^{M \times d})$ and $\mathfrak{m}_U^{(\mathbf{lr})}:  L^2(\mathbb{R}; \mathbb{C}^{d \times N}) \to  H^{\frac{1}{2}}_+(\mathbb{R}; \mathbb{C}^{d \times N})$ are of finite rank, 
\begin{equation}\label{proMlrrlIntro}
\begin{split}
&\mathfrak{m}_U^{(\mathbf{rl})} : G \in L^2(\mathbb{R}; \mathbb{C}^{M \times d}) \mapsto \mathfrak{m}_U^{(\mathbf{rl})}(G) := U \widehat{U^*G}(0) \in  H^{\frac{1}{2}}_+(\mathbb{R}; \mathbb{C}^{M \times d});\\
&\mathfrak{m}_U^{(\mathbf{lr})}: F \in L^2(\mathbb{R}; \mathbb{C}^{d \times N}) \mapsto \mathfrak{m}_U^{(\mathbf{lr})}(F) := \widehat{F U^*}(0)U \in  H^{\frac{1}{2}}_+(\mathbb{R}; \mathbb{C}^{d \times N}). 
\end{split}
\end{equation}For any $t \in \mathbb{R}$, we define that
\begin{equation}\label{LlrLrlIntro}
\begin{split}
&\mathscr{L}_U^{(\mathbf{rl})}(t): = \frac{1}{2 \pi}\int_0^t e^{-i \tau \mathbf{H}^{(\mathbf{r})}_{U} \mathbf{H}^{(\mathbf{l})}_{U}} \mathfrak{m}_U^{(\mathbf{rl})} e^{i \tau \mathbf{H}^{(\mathbf{r})}_{U} \mathbf{H}^{(\mathbf{l})}_{U}}\mathrm{d}\tau  \in \mathcal{B}_{\mathbb{C}}(L^2_+(\mathbb{R}; \mathbb{C}^{M \times d}));\\
&\mathscr{L}_U^{(\mathbf{lr})}(t): = \frac{1}{2 \pi}\int_0^t e^{-i \tau \mathbf{H}^{(\mathbf{l})}_{U} \mathbf{H}^{(\mathbf{r})}_{U}} \mathfrak{m}_U^{(\mathbf{lr})} e^{i \tau \mathbf{H}^{(\mathbf{l})}_{U} \mathbf{H}^{(\mathbf{r})}_{U}}\mathrm{d}\tau  \in \mathcal{B}_{\mathbb{C}}(L^2_+(\mathbb{R}; \mathbb{C}^{d \times N})).
\end{split}
\end{equation}Then $\mathscr{L}_U^{(\mathbf{rl})}(t)$ and $\mathscr{L}_U^{(\mathbf{lr})}(t)$ are both positive operators. The main result of this paper is stated as follows.
\begin{thm}\label{MSzegoRExp}
Given $s> \frac{1}{2}$ and $M,N \in \mathbb{N}_+$, if $U \in C^{\infty }\left(\mathbb{R};  H^s_+(\mathbb{R}; \mathbb{C}^{M\times N}) \right)$ solves the matrix Szeg\H{o} equation \eqref{MSzego} with $U(0) = U_0  \in  H^s_+(\mathbb{R};  \mathbb{C}^{M\times N})$,  then the Poisson integral  of $U(t)$ is given by
\begin{equation}\label{expforMSzegoR}
\begin{split}
\underline{U}(t, z)= &\tfrac{1}{2\pi i} \mathscr{I} \left(\left( \mathbf{G} +  \mathscr{L}^{(\mathbf{rl})}_{U_0} (t)- z\right)^{-1} e^{-i t \mathbf{H}^{(\mathbf{r})}_{U_0} \mathbf{H}^{(\mathbf{l})}_{U_0}}  (U_0)\right)\\
  = &  \tfrac{1}{2\pi i} \mathscr{I}  \left(\left( \mathbf{G} + \mathscr{L}^{(\mathbf{lr})}_{U_0} (t)- z\right)^{-1} e^{-i t \mathbf{H}^{(\mathbf{l})}_{U_0} \mathbf{H}^{(\mathbf{r})}_{U_0}}  (U_0)\right) \in \mathbb{C}^{M \times N},  
\end{split}
\end{equation}for any $(t, z) \in \mathbb{R} \times \mathbb{C}_+$, where $\mathscr{I}$ is given by  \eqref{IntOpforR}.
\end{thm}
\begin{rem}
Since $\mathscr{L}^{(\mathbf{rl})}_{U_0} (t) \geq 0$, $\mathscr{L}^{(\mathbf{lr})}_{U_0} (t) \geq 0$,   both  $-i \mathscr{L}^{(\mathbf{rl})}_{U_0} (t)$ and $-i \mathscr{L}^{(\mathbf{lr})}_{U_0} (t)$ are bounded dissipative (also accretive)
perturbations of the infinitesial generator $-i \mathbf{G}: \mathrm{Dom}(\mathbf{G})^{M\times N} \to L^2_+(\mathbb{R}; \mathbb{C}^{M \times N})$ of the contraction semigroup $(\mathtt{S}(\eta)^*)_{\eta \geq 0}$. Then theorem 3.7 of Davies \cite{Davies1980} yields that both $-i \mathbf{G} -i \mathscr{L}^{(\mathbf{rl})}_{U_0} (t)$ and $-i \mathbf{G} -i \mathscr{L}^{(\mathbf{lr})}_{U_0} (t)$ are generators of some   contraction semigroups. According to Hille--Yosida theorem (theorem \uppercase\expandafter{\romannumeral10}. 47a and theorem \uppercase\expandafter{\romannumeral10}. 48 of Reed--Simon \cite{Reed-Simon2}), $-i \mathbf{G} -i \mathscr{L}^{(\mathbf{rl})}_{U_0} (t)$ and $-i \mathbf{G} -i \mathscr{L}^{(\mathbf{lr})}_{U_0} (t)$ are both maximal dissipative operators whose domains of definition are $\mathrm{Dom}(\mathbf{G})^{M\times N}$. For any $z \in \mathbb{C}_+$, the operators $ \mathbf{G} +  \mathscr{L}^{(\mathbf{rl})}_{U_0} (t)- z $ and $ \mathbf{G} +  \mathscr{L}^{(\mathbf{lr})}_{U_0} (t)- z $ are both invertible from $\mathrm{Dom}(\mathbf{G})^{M\times N}$ to $L^2_+(\mathbb{R}; \mathbb{C}^{M\times N}) $. Since 
\begin{equation}\label{Ran(G+L-z-)1}
\mathrm{Ran}\left( ( \mathbf{G} +  \mathscr{L}^{(\mathbf{rl})}_{U_0} (t)- z )^{-1}\right)  = \mathrm{Ran}\left( ( \mathbf{G} +  \mathscr{L}^{(\mathbf{rl})}_{U_0} (t)- z )^{-1}\right) = \mathrm{Dom}(\mathbf{G})^{M\times N},
\end{equation}where $\mathscr{I}$ in \eqref{IntOpforR} can be defined, the explicit formula \eqref{expforMSzegoR} is well defined.
\end{rem}
  
\noindent Then $U(t)$ can be described as the $L^2_+(\mathbb{R}; \mathbb{C}^{M \times N})$-limit of its Poisson integral  $\underline{U}(t)$, when $\mathrm{Im}z \to 0^+$, thanks to formula \eqref{CVPoiIntforR}. The proof of \eqref{expforMSzegoR} relies on the Lax pair structure in proposition $\ref{LaxPairThmR}$ and the conjugation acting method similar to G\'erard--Pushnitski \cite{GerPush2023}. A unitary group and another family of unitary operators act  simultaneously on the infinitesimal generator $\mathbf{G}$. Then  the Poisson integral of Hamiltonian flow of \eqref{MSzego} can be linearized by identifying these two actions.\\
 
\noindent This paper is organized as follows. The matrix-valued functional spaces and inequalities are recalled in section $\ref{SectPrelim}$. In section $\ref{SectLaxPair}$, we establish the Lax pair structure of the matrix Szeg\H{o} equation \eqref{MSzego} on the real line. Section $\ref{SectExplForR}$  is dedicated to establishing the explicit formula \eqref{expforMSzegoR}  and proving the main theorem $\ref{MSzegoRExp}$.
 
\begin{center}
$Acknowledgments$
\end{center}
The author is grateful to Georgia Institute of Technology for financially supporting the author's research.

\section{Preliminaries}\label{SectPrelim}
\noindent In this section, we give some preliminaries of the matrix-valued functional spaces. Given $p \in [1, +\infty]$,  $s \geq 0$ and $M,N \in \mathbb{N}_+$, a matrix function $U=\left(U_{kj}\right)_{1 \leq k \leq M, 1 \leq j \leq N}$ belongs to $L^p(\mathbb{R}; \mathbb{C}^{M \times N})$ if and only if its $kj$-entry $U_{kj}$ belongs to $L^p(\mathbb{R}; \mathbb{C})$. The Fourier transform of $U \in L^1(\mathbb{R}; \mathbb{C}^{M \times N})$ is given by
\begin{equation}\label{Fouriertr}
\hat{U}(\xi)=\mathscr{F}(U)(\xi) = \int_{\mathbb{R}}U(x)e^{-i x \xi} \mathrm{d}x \in \mathbb{C}^{M \times N}, \quad \forall \xi \in \mathbb{R}.
\end{equation}Equipped with the following inner product,
\begin{equation}\label{InnerPro}
(U, V) \in L^2(\mathbb{R};  \mathbb{C}^{M\times N})^2 \mapsto \langle U, V \rangle_{L^2(\mathbb{R};  \mathbb{C}^{M\times N}) } :=  \int_{\mathbb{R}}\mathrm{tr} \left( U(x) V(x)^{*}\right)\mathrm{d}x \in \mathbb{C},
\end{equation}$L^2(\mathbb{R};  \mathbb{C}^{M\times N})$ is a $\mathbb{C}$-Hilbert space. The Plancherel's theorem shows that the renormalized Fourier transform $(2\pi)^{-\frac{1}{2}}\mathscr{F}$ can be extended to a $\mathbb{C}$-Hilbert isomorphism on $L^2(\mathbb{R};  \mathbb{C}^{M\times N})$ and we have
\begin{equation}\label{Plancherel}
\langle U, V \rangle_{L^2(\mathbb{R};  \mathbb{C}^{M\times N}) } = \tfrac{1}{2\pi} \langle \hat{U}, \hat{V} \rangle_{L^2(\mathbb{R};  \mathbb{C}^{M\times N}) }, \quad \forall U, V \in L^2(\mathbb{R};  \mathbb{C}^{M\times N}).
\end{equation}A matrix function $U=\left(U_{kj}\right)_{1 \leq k \leq M, 1 \leq j \leq N}$ belongs to $H^s(\mathbb{R}; \mathbb{C}^{M \times N})$ if and only if its $kj$-entry $U_{kj}$ belongs to the Sobolev space $H^s(\mathbb{R}; \mathbb{C})$, $\forall s \in \mathbb{R}$. Equipped with the following inner product
\begin{equation}\label{InnerProHs}
(U, V) \in H^s(\mathbb{R};  \mathbb{C}^{M\times N})^2 \mapsto \langle U, V \rangle_{H^s(\mathbb{R};  \mathbb{C}^{M\times N}) } :=  \tfrac{1}{2 \pi} \int_{\mathbb{R}}\mathrm{tr} \left( \hat{U}(\xi) \hat{V}(\xi)^{*}\right)(1+ |\xi|^2)^s \mathrm{d} \xi \in \mathbb{C},
\end{equation}the matrix-valued Sobolev space $H^s(\mathbb{R};  \mathbb{C}^{M\times N})$ is a $\mathbb{C}$-Hilbert space. We set
\begin{equation}\label{LpMat}
\begin{split}
& \|U\|_{L^p(\mathbb{R}; \mathbb{C}^{M \times N})}^2 := \sum_{k=1}^M \sum_{j=1}^N \|U_{kj}\|_{L^p(\mathbb{R}; \mathbb{C})}^2; \quad 
 \|U\|_{H^s(\mathbb{R}; \mathbb{C}^{M \times N})}^2 := \sum_{k=1}^M \sum_{j=1}^N \|U_{kj}\|_{H^s(\mathbb{R}; \mathbb{C})}^2 .
\end{split}
\end{equation}If $p^{-1}+q^{-1} = r^{-1}$ for some $p,q,r \geq 1$, we have the following H\"older's inequality
\begin{equation}\label{Hoelder}
\|UV\|_{L^r(\mathbb{R}; \mathbb{C}^{M \times P})} \leq \|U\|_{L^p(\mathbb{R}; \mathbb{C}^{M \times N})} \|V\|_{L^q(\mathbb{R}; \mathbb{C}^{N \times P})}, \quad \forall U \in L^p(\mathbb{R}; \mathbb{C}^{M \times N}), \; V \in L^q(\mathbb{R}; \mathbb{C}^{N \times P}).
\end{equation}The transpose transform $\mathfrak{T}$ in \eqref{Transpose} preserves every $L^p$-norm and $H^s$-norm, $\forall p \geq 1$, $\forall s \in \mathbb{R}$, i.e.
\begin{equation}
\|U^{\mathrm{T}}\|_{L^p(\mathbb{R}; \mathbb{C}^{N  \times M})} = \|U\|_{L^p(\mathbb{R}; \mathbb{C}^{M \times N})}, \quad \|U^{\mathrm{T}}\|_{H^s(\mathbb{R}; \mathbb{C}^{N  \times M})} = \|U\|_{H^s (\mathbb{R}; \mathbb{C}^{M \times N})}
\end{equation}The negative Szeg\H{o} projector  $\Pi_{< 0} = \mathrm{id}_{L^2(\mathbb{R};  \mathbb{C}^{M\times N})} - \Pi_{\geq 0} = \mathbf{1}_{(-\infty, 0)}(\mathrm{D})$ on $L^2(\mathbb{R};  \mathbb{C}^{M\times N})$ is given by
\begin{equation}\label{Szego-}
\widehat{\Pi_{<0}(U)} (\xi_1) =0, \quad  \widehat{\Pi_{<0}(U)} (\xi_2) = \hat{U}(\xi_2), \quad \forall \xi_1 > 0 > \xi_2, \quad \forall U \in L^2(\mathbb{R};  \mathbb{C}^{M\times N}).
 \end{equation}The filtered matrix-valued Sobolev spaces are given by 
 \begin{equation}
 H^s_+(\mathbb{R};  \mathbb{C}^{M\times N}):=  \Pi_{\geq 0}\left(H^s(\mathbb{T};  \mathbb{C}^{M\times N}) \right), \quad  H^s_-(\mathbb{R};  \mathbb{C}^{M\times N}):=  \Pi_{<0}\left(H^s(\mathbb{T};  \mathbb{C}^{M\times N}) \right), \quad \forall s \geq 0 .
 \end{equation}Then we have the following orthogonal decomposition for the $\mathbb{C}$-Hilbert space $ L^2(\mathbb{R};  \mathbb{C}^{M\times N})$,
 \begin{equation}
 L^2(\mathbb{R};  \mathbb{C}^{M\times N}) = L^2_+(\mathbb{R};  \mathbb{C}^{M\times N}) \bigoplus L^2_-(\mathbb{R};  \mathbb{C}^{M\times N}), \quad L^2_+(\mathbb{R};  \mathbb{C}^{M\times N}) \perp L^2_-(\mathbb{R};  \mathbb{C}^{M\times N}).
 \end{equation}If $U \in L^2(\mathbb{R};  \mathbb{C}^{M \times N})$, then $\Pi_{< 0} U = \left(\Pi_{\geq 0} (U^*) \right)^* \in L^2_-(\mathbb{R};  \mathbb{C}^{M \times N})$ and
\begin{equation}\label{Pi<0FR}
U= \Pi_{\geq 0} U + \Pi_{< 0} U = \Pi_{\geq 0} U + \left(\Pi_{\geq 0} (U^*) \right)^*.
\end{equation}

\begin{lem}
If $U \in L^2(\mathbb{R}; \mathbb{C}^{M\times N})$, $A\in \mathbb{C}^{N \times P}$, $B\in \mathbb{C}^{Q\times M}$ for some $M,N,P,Q \in \mathbb{N}_+$, then
\begin{equation}\label{matrixXL^2}
\begin{split}
& \Pi_{\geq 0} (UA) = \left(\Pi_{\geq 0}U \right)A \in L^2_+ (\mathbb{R}; \mathbb{C}^{M \times P}); \quad \Pi_{< 0} (UA) = \left(\Pi_{< 0}U\right)A \in L^2_- (\mathbb{R}; \mathbb{C}^{M \times P}); \\
& \Pi_{\geq 0} (BU) = B \left(\Pi_{\geq 0}U \right) \in  L^2_+ (\mathbb{R}; \mathbb{C}^{Q \times N}); \quad  \Pi_{< 0} (BU) = B \left(\Pi_{< 0}U \right) \in  L^2_- (\mathbb{R}; \mathbb{C}^{Q \times N}).
\end{split}
\end{equation}
\end{lem}
\begin{proof}
See lemma 2.1 of Sun \cite{SUNMSzego}.
\end{proof}
\begin{lem}\label{ABL2+forR}
Given $A_+ \in L^2_+(\mathbb{R}; \mathbb{C}^{M\times N})$ and $B_+ \in L^2_+(\mathbb{R}; \mathbb{C}^{N \times d})$ for some $M,N, d \in \mathbb{N}_+$, if one of $A_+, B_+$ is essentially bounded , then $A_+  B_+ \in L^2_+(\mathbb{R}; \mathbb{C}^{M \times d})$. Given $A_- \in L^2_-(\mathbb{R}; \mathbb{C}^{M\times N})$ and $B_- \in L^2_-(\mathbb{R}; \mathbb{C}^{N \times d})$, if one of $A_-, B_-$ is essentially bounded , then $A_-  B_- \in L^2_-(\mathbb{R}; \mathbb{C}^{M \times d})$.
\end{lem}
\begin{proof}
See lemma 2.3 and corollary 2.4 of Sun \cite{SUNInterCMSDNLS}.
\end{proof}
\noindent Inspired from \cite{poDCDS2011}, \cite{BOExp2022}, \cite{GerPush2023} and \cite{SUNInterCMSDNLS}, we introduce an  auxiliary approximate identity
\begin{equation}\label{chieps}
\chi_{\epsilon}(x) =\tfrac{2\pi}{\epsilon}\Pi_{\geq 0} \left(\mathfrak{P}_{\epsilon^{-1}} \right) (x) =  (1-i\epsilon x)^{-1}, \quad \forall 0 < \epsilon <1,
\end{equation}where $\mathfrak{P}_y$ is the Poisson kernel defined in \eqref{PoiIntforR}. Then $\chi_{\epsilon} \in H_+^{+\infty}(\mathbb{R}; \mathbb{C}):=\bigcap_{s\geq 0} H^s_+(\mathbb{R}; \mathbb{C})$. Let $\mathbb{E}_{kj}^{(M N)}$ denote the $M\times N$ matrix whose $kj$-entry is $1$, whose other entries are all $0$. For any $p\in \mathbb{N}$, we have
\begin{equation}\label{CVdominesca}
\lim_{\epsilon \to 0^+}\|\chi_{\epsilon} U -U\|_{H^p(\mathbb{R}; \mathbb{C}^{M \times N})}=0, \quad \lim_{\epsilon \to 0^+} \langle F, \chi_{\epsilon} \mathbb{E}^{(MN)}_{kj} \rangle_{L^2_+(\mathbb{R}; \mathbb{C}^{M \times N})} = \hat{F}_{kj}(0^+),
\end{equation}$\forall U \in H^p(\mathbb{R}; \mathbb{C}^{M \times N})$, $\forall F \in L^2_+(\mathbb{R}; \mathbb{C}^{M\times N})$ such that $\hat{F}_{kj}$ is right continuous at $0$, by using Lebesgue's dominated convergence theorem. Then  $\mathscr{I}(F)$ in \eqref{IntOpforR} can be expressed in terms of the $L^2_+(\mathbb{R}; \mathbb{C}^{M \times N})$-inner products of $F$ and the auxiliary functions $\chi_{\epsilon} \mathbb{E}^{(MN)}_{kj} \in H_+^{+\infty}(\mathbb{R}; \mathbb{C}^{M \times N})$,
\begin{equation}\label{HATF0POsitiv}
\mathscr{I}(F) = \hat{F}(0^+)=\sum_{k=1}^M \sum_{j=1}^N \lim_{\epsilon \to 0^+} \langle F, \chi_{\epsilon}   \mathbb{E}^{(MN)}_{kj} \rangle_{L^2_+(\mathbb{R}; \mathbb{C}^{M \times N})}\mathbb{E}^{(MN)}_{kj} \in \mathbb{C}^{M \times N}.
\end{equation}The Poisson integral $\underline{U}$ defined in \eqref{PoiIntforR} can be expressed in terms of the resolvent of $\mathbf{G}$ and the original function $U \in L^2_+(\mathbb{R}; \mathbb{C}^{M \times N})$, thanks to the next lemma.
\begin{lem}[Sun \cite{SUNInterCMSDNLS}]\label{LeminvforR}
Given any $z\in \mathbb{C}_+$, $M,N \in \mathbb{N}_+$, if $U \in L^2_+(\mathbb{R}; \mathbb{C}^{M \times N})$, then  
\begin{equation}\label{invForR}
\underline{U}(z) = \tfrac{1}{2\pi i} \mathscr{I}\left((\mathbf{G}-z)^{-1}(U) \right) = \tfrac{1}{2\pi i}\sum_{k=1}^M \sum_{j=1}^N \lim_{\epsilon \to 0^+} \langle (\mathbf{G}-z)^{-1}(U), \chi_{\epsilon}\mathbb{E}^{(MN)}_{kj} \rangle_{L^2_+(\mathbb{R}; \mathbb{C}^{M \times N})} \mathbb{E}^{(MN)}_{kj}.
\end{equation} 
\end{lem}
\begin{proof}
See lemma 4.2 of Sun \cite{SUNInterCMSDNLS}.
\end{proof}
\noindent The transpose transform for matrices, denoted by
\begin{equation}\label{Transpose}
\mathfrak{T}=\mathfrak{T}^{-1}: A \in \mathbb{C}^{M \times N} \mapsto \mathfrak{T}(A) = A^{\mathrm{T}} \in \mathbb{C}^{N \times M},
\end{equation}is a self-adjoint unitary operator on every Sobolev space $H^s(\mathbb{R};  \mathbb{C}^{M\times N})$ and on every filtered Sobolev space $H^s_+(\mathbb{R};  \mathbb{C}^{M\times N}) $, $\forall s \in \mathbb{R}$. According to Sun \cite{SUNMSzego},  the matrix Szeg\H{o} equation on the line \eqref{MSzego} is invariant under transposing, i.e. if  $U \in C^{\infty}(\mathbb{R}; H^s_+(\mathbb{R};  \mathbb{C}^{M\times N}))$ solves \eqref{MSzego}, so does $U^{\mathrm{T}} \in C^{\infty}(\mathbb{R}; H^s_+(\mathbb{R};  \mathbb{C}^{N\times M}))$, $\forall s \geq \frac{1}{2}$. Moreover, $\mathfrak{T}$ in \eqref{Transpose} commutes with the Fourier transform $\mathscr{F}$ and all Fourier multipliers,  
\begin{equation}\label{TF=FT}
\mathfrak{T} \mathscr{F} = \mathscr{F} \mathfrak{T} , \quad \mathfrak{T} \Pi_{\geq 0} =  \Pi_{\geq 0} \mathfrak{T}, \quad \mathfrak{T} \Pi_{< 0} =  \Pi_{< 0} \mathfrak{T}, \quad \mathfrak{T} \partial_x =  \partial_x  \mathfrak{T}, \quad \mathfrak{T} \mathscr{I}_+ =  \mathscr{I}_+  \mathfrak{T}, \quad \mathfrak{T} \mathscr{P} =\mathscr{P} \mathfrak{T} .
\end{equation}The transpose transform $\mathfrak{T}$ also commutes with the Lax--Beurling semigroups $(\mathtt{S}(\eta))_{\eta \geq 0}$ and $(\mathtt{S}(\eta)^*)_{\eta \geq 0}$. Moreover,   $\mathfrak{T}\left(\mathrm{Dom}(\mathbf{G})^{M\times N}\right) = \mathrm{Dom}(\mathbf{G})^{N \times M}$ and $\mathfrak{T}$ commutes with $\mathbf{G}$, i.e.
\begin{equation}\label{TG=GT}
\mathfrak{T} \mathbf{G} = \mathbf{G} \mathfrak{T} , \quad \mathfrak{T} \mathtt{S}(\eta) = \mathtt{S}(\eta)\mathfrak{T} , \quad \mathfrak{T} \mathtt{S}(\eta)^* = \mathtt{S}(\eta)^* \mathfrak{T}.
\end{equation}At last, we recall a theorem of Cauchy.
 \begin{prop}[Cauchy]\label{CauchyThmODE}
Let $\mathcal{E}$ be a Banach space, $\mathcal{I}$ is an open interval of $\mathbb{R}$ and $A \in C^0 (\mathcal{I}; \mathcal{B}(\mathcal{E}))$, if $(t_0, x_0) \in \mathcal{I} \times \mathcal{E}$, there exists a unique function $x \in C^1(\mathcal{I}; \mathcal{B}(\mathcal{E}))$ which solves  
\begin{equation}
 x(t_0)=x_0, \quad  \tfrac{\mathrm{d}}{\mathrm{d}t} x(t) = A(t)\left( x(t)\right), \quad \forall t \in \mathcal{I}.
\end{equation} 
\end{prop}
\begin{proof}
See Theorem 1.1.1 of Chemin \cite{CheminNoteEDPM2-2016}.
\end{proof}
  
\section{The Lax pair structure}\label{SectLaxPair}
\noindent This section is devoted to proving proposition $\ref{LaxPairThmR}$. 
\subsection{The Hankel operators}
\noindent Given any $d, M,N\in\mathbb{N}_+$,  the Hankel operators of symbol $U \in H^{\frac{1}{2}}_+(\mathbb{R}; \mathbb{C}^{M\times N})$ are given by
\begin{equation}\label{rlHankel}
\begin{split}
& \mathbf{H}^{(\mathbf{r})}_{U} : F \in L^2_+(\mathbb{R}; \mathbb{C}^{d \times N}) \mapsto \mathbf{H}^{(\mathbf{r})}_{U}(F) = \Pi_{\geq 0}(U F^*) \in L^2_+(\mathbb{R}; \mathbb{C}^{M \times d});\\
 & \mathbf{H}^{(\mathbf{l})}_{U} : G \in L^2_+(\mathbb{R}; \mathbb{C}^{M \times d}) \mapsto \mathbf{H}^{(\mathbf{l})}_{U}(G) = \Pi_{\geq 0}(G^* U) \in L^2_+(\mathbb{R}; \mathbb{C}^{d \times N}).\\
\end {split}
\end{equation}If $F=\sum_{j=1}^d \sum_{n=1}^N F_{jn}\mathbb{E}^{(dN)}_{jn} \in  L^2_+(\mathbb{R}; \mathbb{C}^{d \times N})$ and $G=\sum_{m=1}^M \sum_{k=1}^d G_{mk}\mathbb{E}^{(Md)}_{mk} \in  L^2_+(\mathbb{R}; \mathbb{C}^{M \times d})$, then
\begin{equation}\label{relHrHltoH}
\mathbf{H}^{(\mathbf{r})}_{U}(F) = \sum_{k=1}^M  \sum_{j=1}^d (\sum_{n=1}^N H_{U_{kn}}(F_{jn}))\mathbb{E}^{(Md)}_{kj}; \quad \mathbf{H}^{(\mathbf{l})}_{U}(G) = \sum_{k=1}^d  \sum_{j=1}^N (\sum_{m=1}^M H_{U_{mj}}(G_{mk}))\mathbb{E}^{(dN)}_{kj}. 
\end{equation}Both $\mathbf{H}^{(\mathbf{r})}_{U}: L^2_+(\mathbb{R}; \mathbb{C}^{d \times N}) \to L^2_+(\mathbb{R}; \mathbb{C}^{M \times d})$ and $\mathbf{H}^{(\mathbf{l})}_{U}: L^2_+(\mathbb{R}; \mathbb{C}^{M \times d}) \to  L^2_+(\mathbb{R}; \mathbb{C}^{d \times N})$ are $\mathbb{C}$-antilinear Hilbert--Schmidt operators by Lemma 3.5 of Pocovnicu \cite{poAPDE2011}. Moreover,  
\begin{equation}\label{HrHlinnprod}
\langle  \mathbf{H}^{(\mathbf{r})}_{U}(F), G \rangle_{L^2_+(\mathbb{R}; \mathbb{C}^{M \times d})} =  \int_{\mathbb{R}}\mathrm{tr}\left(G(x)^* U(x) F(x)^*  \right)\mathrm{d}x = \langle  \mathbf{H}^{(\mathbf{l})}_{U}(G), F \rangle_{L^2_+(\mathbb{R}; \mathbb{C}^{d \times N})} 
\end{equation}holds for any $F \in L^2_+(\mathbb{R}; \mathbb{C}^{d \times N})$ and $G \in L^2_+(\mathbb{R}; \mathbb{C}^{M \times d})$. Moreover, the right Hankel operators are conjugate to the left Hankel operators via the transpose transform \eqref{Transpose},  
\begin{equation}\label{TconjHrtoHl}
\mathbf{H}^{(\mathbf{l})}_{U} = \mathfrak{T}  \mathbf{H}^{(\mathbf{r})}_{U^{\mathrm{T}}}  \mathfrak{T}, \quad \mathbf{H}^{(\mathbf{r})}_{U} = \mathfrak{T}  \mathbf{H}^{(\mathbf{l})}_{U^{\mathrm{T}}}  \mathfrak{T}.
\end{equation}
 
\noindent If $U \in H^{\frac{1}{2}}_+ \bigcup L^{\infty}(\mathbb{R}; \mathbb{C}^{M \times N})$, $V \in H^{\frac{1}{2}}_+ \bigcup L^{\infty}(\mathbb{R}; \mathbb{C}^{P \times N})$ and $W \in H^{\frac{1}{2}}_+ \bigcup L^{\infty} (\mathbb{R}; \mathbb{C}^{P\times Q})$ for some positive integers $M,N, P, Q\in\mathbb{N}_+$, the matrices $U(x) \in \mathbb{C}^{M \times N}$ and  $V(x) \in \mathbb{C}^{P \times N}$ have the same number of columns,  $\forall x\in \mathbb{R}$, then the $(\mathbf{rl})$-double Hankel operators are defined as
\begin{equation}\label{HrHldef}
\begin{split}
  \mathbf{H}^{(\mathbf{r})}_{U} \mathbf{H}^{(\mathbf{l})}_{V}: G_1 \in  L^2_+(\mathbb{R}; \mathbb{C}^{P \times d}) \mapsto\mathbf{H}^{(\mathbf{r})}_{U} \mathbf{H}^{(\mathbf{l})}_{V}(G_1) = \Pi_{\geq 0}\left(U \left(\Pi_{\geq 0 }(G_1^* V)\right)^*\right)\in L^2_+(\mathbb{R}; \mathbb{C}^{M \times d}); \\
\mathbf{H}^{(\mathbf{r})}_{V} \mathbf{H}^{(\mathbf{l})}_{U}: G_2 \in L^2_+(\mathbb{R}; \mathbb{C}^{M \times d})  \mapsto \mathbf{H}^{(\mathbf{r})}_{V} \mathbf{H}^{(\mathbf{l})}_{U}(G_2) = \Pi_{\geq 0}\left(V \left(\Pi_{\geq 0 }(G_2^* U)\right)^*\right)\in L^2_+(\mathbb{R}; \mathbb{C}^{P \times d}).
\end{split}
\end{equation}for any $d \in \mathbb{N}_+$. The matrices $V(x) \in \mathbb{C}^{P \times N}$ and $W(x) \in \mathbb{C}^{P \times Q}$ have the same number of rows, the $(\mathbf{lr})$-double Hankel operators are given by
\begin{equation}\label{HlHrdef}
\begin{split}
\mathbf{H}^{(\mathbf{l})}_{V} \mathbf{H}^{(\mathbf{r})}_{W} : F_1 \in L^2_+(\mathbb{R}; \mathbb{C}^{d \times Q}) \mapsto \mathbf{H}^{(\mathbf{l})}_{V} \mathbf{H}^{(\mathbf{r})}_{W}(F_1) = \Pi_{\geq 0}\left(\left(\Pi_{\geq 0 }(W F_1^*)\right)^* V\right) \in L^2_+(\mathbb{R}; \mathbb{C}^{d \times N});\\
\mathbf{H}^{(\mathbf{l})}_{W} \mathbf{H}^{(\mathbf{r})}_{V} : F_2 \in L^2_+(\mathbb{R}; \mathbb{C}^{d \times N}) \mapsto \mathbf{H}^{(\mathbf{l})}_{W} \mathbf{H}^{(\mathbf{r})}_{V}(F_2) = \Pi_{\geq 0}\left(\left(\Pi_{\geq 0 }(V F_2^*)\right)^* W\right) \in L^2_+(\mathbb{R}; \mathbb{C}^{d \times Q}).
\end{split}
\end{equation}If $G_1 \in L^2_+(\mathbb{R}; \mathbb{C}^{P \times d}) $, $G_2 \in L^2_+(\mathbb{R}; \mathbb{C}^{M \times d})$, $F_1 \in L^2_+(\mathbb{R}; \mathbb{C}^{d \times Q})$ and $F_2 \in L^2_+(\mathbb{R}; \mathbb{C}^{d \times N})$, we have 
\begin{equation}\label{HrHladj}
\small
\begin{cases}
\langle \mathbf{H}^{(\mathbf{r})}_{U} \mathbf{H}^{(\mathbf{l})}_{V} (G_1), G_2 \rangle_{ L^2_+(\mathbb{R}; \mathbb{C}^{M \times d})} = \langle \mathbf{H}^{(\mathbf{l})}_{U}(G_2) , \mathbf{H}^{(\mathbf{l})}_{V} (G_1) \rangle_{ L^2_+(\mathbb{R}; \mathbb{C}^{d \times N})} = \langle  G_1 , \mathbf{H}^{(\mathbf{r})}_{V} \mathbf{H}^{(\mathbf{l})}_{U}(G_2) \rangle_{ L^2_+(\mathbb{R}; \mathbb{C}^{P \times d})};\\
\langle \mathbf{H}^{(\mathbf{l})}_{V} \mathbf{H}^{(\mathbf{r})}_{W} (F_1), F_2 \rangle_{ L^2_+(\mathbb{R}; \mathbb{C}^{d \times N})} = \langle \mathbf{H}^{(\mathbf{r})}_{V} (F_2) ,  \mathbf{H}^{(\mathbf{r})}_{W} (F_1) \rangle_{ L^2_+(\mathbb{R}; \mathbb{C}^{P \times d})} = \langle  F_1 , \mathbf{H}^{(\mathbf{l})}_{W} \mathbf{H}^{(\mathbf{r})}_{V} (F_2) \rangle_{ L^2_+(\mathbb{R}; \mathbb{C}^{d \times Q})};
\end{cases}
\end{equation}by using formula \eqref{HrHlinnprod}. Indeed, $(\mathbf{rl})$-double Hankel operators are conjugated to $(\mathbf{lr})$-double Hankel operators via the transpose transform \eqref{Transpose}, i.e.
\begin{equation}\label{TconjHH}
\mathbf{H}^{(\mathbf{r})}_{U} \mathbf{H}^{(\mathbf{l})}_{V} = \mathfrak{T} \mathbf{H}^{(\mathbf{l})}_{U^{\mathrm{T}}} \mathbf{H}^{(\mathbf{r})}_{V^{\mathrm{T}}} \mathfrak{T};  \quad \mathbf{H}^{(\mathbf{l})}_{V} \mathbf{H}^{(\mathbf{r})}_{W} = \mathfrak{T} \mathbf{H}^{(\mathbf{l})}_{V^{\mathrm{T}}} \mathbf{H}^{(\mathbf{r})}_{W^{\mathrm{T}}} \mathfrak{T}.
\end{equation}If $M=P$, then $\mathbf{H}^{(\mathbf{r})}_{U} \mathbf{H}^{(\mathbf{l})}_{V}$ is a $\mathbb{C}$-linear bounded operator on $L^2_+(\mathbb{R}; \mathbb{C}^{M \times d})$ and $\mathbf{H}^{(\mathbf{r})}_{V} \mathbf{H}^{(\mathbf{l})}_{U} = \left( \mathbf{H}^{(\mathbf{r})}_{U} \mathbf{H}^{(\mathbf{l})}_{V} \right)^*$. So  $\mathbf{H}^{(\mathbf{r})}_{U} \mathbf{H}^{(\mathbf{l})}_{U}\geq 0$ is a $\mathbb{C}$-linear bounded positive self-adjoint operator on $L^2_+(\mathbb{R}; \mathbb{C}^{M \times d})$. Similarly, if $N=Q$, then $\mathbf{H}^{(\mathbf{l})}_{V} \mathbf{H}^{(\mathbf{r})}_{W}$ is a $\mathbb{C}$-linear bounded operator on $L^2_+(\mathbb{R}; \mathbb{C}^{d \times Q})$ and $\mathbf{H}^{(\mathbf{l})}_{W} \mathbf{H}^{(\mathbf{r})}_{V} = \left( \mathbf{H}^{(\mathbf{l})}_{V} \mathbf{H}^{(\mathbf{r})}_{W} \right)^*$. So $\mathbf{H}^{(\mathbf{l})}_{W} \mathbf{H}^{(\mathbf{r})}_{W}\geq 0$ is a $\mathbb{C}$-linear positive self-adjoint operator on $L^2_+(\mathbb{R}; \mathbb{C}^{d \times Q})$. Then \eqref{TconjHH} yields that $\left(\mathbf{H}^{(\mathbf{r})}_{U} \mathbf{H}^{(\mathbf{l})}_{U} \right)^n = \mathfrak{T} \left(\mathbf{H}^{(\mathbf{l})}_{U^{\mathrm{T}}} \mathbf{H}^{(\mathbf{r})}_{U^{\mathrm{T}}}\right)^n \mathfrak{T} $ and $\left(\mathbf{H}^{(\mathbf{l})}_{U} \mathbf{H}^{(\mathbf{r})}_{U} \right)^n = \mathfrak{T} \left(\mathbf{H}^{(\mathbf{r})}_{U^{\mathrm{T}}} \mathbf{H}^{(\mathbf{l})}_{U^{\mathrm{T}}}\right)^n \mathfrak{T}$, $\forall n \in \mathbb{N}$. Moreover,\begin{equation}\label{TconjexpHH}
\small
e^{it\mathbf{H}^{(\mathbf{r})}_{U} \mathbf{H}^{(\mathbf{l})}_{U}} = \mathfrak{T} e^{it\mathbf{H}^{(\mathbf{l})}_{U^{\mathrm{T}}} \mathbf{H}^{(\mathbf{r})}_{U^{\mathrm{T}}}} \mathfrak{T} \in \mathcal{B}_{\mathbb{C}}\left(L^2_+(\mathbb{R}; \mathbb{C}^{M \times d}) \right); \quad e^{it\mathbf{H}^{(\mathbf{l})}_{U} \mathbf{H}^{(\mathbf{r})}_{U}} = \mathfrak{T} e^{it\mathbf{H}^{(\mathbf{r})}_{U^{\mathrm{T}}} \mathbf{H}^{(\mathbf{l})}_{U^{\mathrm{T}}}} \mathfrak{T} \in \mathcal{B}_{\mathbb{C}}\left(L^2_+(\mathbb{R}; \mathbb{C}^{d \times N}) \right),
\end{equation}for any $t \in \mathbb{R}$. When $U \in H^{\frac{1}{2}}_+ (\mathbb{R}; \mathbb{C}^{M \times N}), V \in H^{\frac{1}{2}}_+ (\mathbb{R}; \mathbb{C}^{M \times N})$, the double Hankel operator
\begin{equation}\label{HrHltracecl}
\mathbf{H}^{(\mathbf{r})}_{V} \mathbf{H}^{(\mathbf{l})}_{U} = \mathfrak{T} \mathbf{H}^{(\mathbf{l})}_{V^{\mathrm{T}}} \mathbf{H}^{(\mathbf{r})}_{U^{\mathrm{T}}} \mathfrak{T} = \left( \mathbf{H}^{(\mathbf{r})}_{U} \mathbf{H}^{(\mathbf{l})}_{V} \right)^* \in \mathcal{B}_{\mathbb{C}}\left(L^2_+(\mathbb{R}; \mathbb{C}^{M \times d}) \right)
\end{equation}is a trace-class operator.\\
 
\noindent The projection operators in \eqref{proMlrrlIntro} can be defined when the symbol $U \in L^2(\mathbb{R}; \mathbb{C}^{M \times N})$,
\begin{equation}\label{proMlrrl}
\begin{split}
&\mathfrak{m}_U^{(\mathbf{rl})} : G \in L^2(\mathbb{R}; \mathbb{C}^{M \times d}) \mapsto \mathfrak{m}_U^{(\mathbf{rl})}(G) := U \widehat{U^*G}(0) \in  L^2(\mathbb{R}; \mathbb{C}^{M \times d});\\
&\mathfrak{m}_U^{(\mathbf{lr})}: F \in L^2(\mathbb{R}; \mathbb{C}^{d \times N}) \mapsto \mathfrak{m}_U^{(\mathbf{lr})}(F) := \widehat{F U^*}(0)U \in  L^2(\mathbb{R}; \mathbb{C}^{d \times N});\\
\end{split}
\end{equation}Then we have $\max\{\|\mathfrak{m}_U^{(\mathbf{rl})}\|_{\mathcal{B}(L^2(\mathbb{R}; \mathbb{C}^{M \times d}))}, \|\mathfrak{m}_U^{(\mathbf{lr})}\|_{\mathcal{B}(L^2(\mathbb{R}; \mathbb{C}^{d \times N}))}\}\leq \|U\|_{L^2(\mathbb{R}; \mathbb{C}^{M \times N})}^2$ and
\begin{equation}\label{Tconjm(lr)m(rl)}
\mathfrak{m}^{(\mathbf{rl})}_U = \mathfrak{T}\mathfrak{m}_{U^\mathrm{T}}^{(\mathbf{lr})} \mathfrak{T}, \quad \mathfrak{m}^{(\mathbf{lr})}_U = \mathfrak{T}\mathfrak{m}_{U^\mathrm{T}}^{(\mathbf{rl})} \mathfrak{T}.
\end{equation}Recall the definition in \eqref{LlrLrlIntro}, if $U \in H^{\frac{1}{2}}_+(\mathbb{R}; \mathbb{C}^{M \times N})$ and $ (F, G) \in L^2_+(\mathbb{R}; \mathbb{C}^{d \times N}) \times L^2_+(\mathbb{R}; \mathbb{C}^{M \times d})$, 
\begin{equation}\label{LlrLrl}
\begin{split}
&\mathscr{L}_U^{(\mathbf{rl})}(t) (G) = \frac{1}{2 \pi}\int_0^t e^{-i \tau \mathbf{H}^{(\mathbf{r})}_{U} \mathbf{H}^{(\mathbf{l})}_{U}} \mathfrak{m}_U^{(\mathbf{rl})} e^{i \tau \mathbf{H}^{(\mathbf{r})}_{U} \mathbf{H}^{(\mathbf{l})}_{U}}(G)\mathrm{d}\tau  \in  L^2_+(\mathbb{R}; \mathbb{C}^{M \times d}),\\
&\mathscr{L}_U^{(\mathbf{lr})}(t)(F) = \frac{1}{2 \pi}\int_0^t e^{-i \tau \mathbf{H}^{(\mathbf{l})}_{U} \mathbf{H}^{(\mathbf{r})}_{U}} \mathfrak{m}_U^{(\mathbf{lr})} e^{i \tau \mathbf{H}^{(\mathbf{l})}_{U} \mathbf{H}^{(\mathbf{r})}_{U}}(F)\mathrm{d}\tau  \in  L^2_+(\mathbb{R}; \mathbb{C}^{d \times N}),
\end{split}
\end{equation}are interpreted as the vector-valued integrations in definition 3.26 of Rudin \cite{RudinFA}. Theorem 3.29 of  Rudin \cite{RudinFA} yields that $\max\{\|\mathscr{L}_U^{(\mathbf{rl})}(t)\|_{\mathcal{B}(L^2(\mathbb{R}; \mathbb{C}^{M \times d}))}, \|\mathscr{L}_U^{(\mathbf{lr})}(t)\|_{\mathcal{B}(L^2(\mathbb{R}; \mathbb{C}^{d \times N}))} \}\leq \tfrac{|t|}{2 \pi}\|U\|_{L^2(\mathbb{R}; \mathbb{C}^{M \times N})}^2$. We have
\begin{equation}\label{TconjL(lr)L(rl)}
\mathscr{L}^{(\mathbf{rl})}_U (t) = \mathfrak{T}\mathscr{L}_{U^\mathrm{T}}^{(\mathbf{lr})} (t)  \mathfrak{T} \in \mathcal{B}_{\mathbb{C}}(L^2_+(\mathbb{R}; \mathbb{C}^{M \times d})), \quad \mathscr{L}^{(\mathbf{lr})}_U (t)  = \mathfrak{T}\mathscr{L}_{U^\mathrm{T}}^{(\mathbf{rl})} (t) \mathfrak{T} \in \mathcal{B}_{\mathbb{C}}(L^2_+(\mathbb{R}; \mathbb{C}^{d \times N})),
\end{equation}thanks to \eqref{TconjexpHH} and \eqref{Tconjm(lr)m(rl)}. For any $z \in \mathbb{C}_+$, the operators $ \mathbf{G} + \mathscr{L}^{(\mathbf{rl})}_U (t) -z  \in \mathcal{B}_{\mathbb{C}}(L^2_+(\mathbb{R}; \mathbb{C}^{M \times d}))$ and $ \mathbf{G} + \mathscr{L}^{(\mathbf{lr})}_U (t) -z  \in \mathcal{B}_{\mathbb{C}}(L^2_+(\mathbb{R}; \mathbb{C}^{d \times N}))$ are both invertible and we have
\begin{equation}\label{TconjinvG+Lrl-z}
\small
\left(\mathbf{G} + \mathscr{L}^{(\mathbf{rl})}_U (t) -z \right)^{-1} = \mathfrak{T}\left(\mathbf{G} + \mathscr{L}^{(\mathbf{lr})}_{U^{\mathrm{T}}} (t) -z \right)^{-1} \mathfrak{T}, \quad \left(\mathbf{G} + \mathscr{L}^{(\mathbf{lr})}_U (t) -z \right)^{-1} = \mathfrak{T}\left(\mathbf{G} + \mathscr{L}^{(\mathbf{rl})}_{U^{\mathrm{T}}} (t) -z \right)^{-1} \mathfrak{T}.
\end{equation}

\subsection{The Toeplitz operators}
\noindent Recall the definition of Toeplitz operators in \eqref{rlToepIntro}, given $d, M,N\in\mathbb{N}_+$ and $V \in L^2(\mathbb{R}; \mathbb{C}^{M \times N})$,
\begin{equation}\label{rlToep}
\begin{split}
& \mathbf{T}^{(\mathbf{r})}_{V} : G \in H^1_+(\mathbb{R}; \mathbb{C}^{N \times d}) \mapsto \mathbf{T}^{(\mathbf{r})}_{V}(G) = \Pi_{\geq 0}(V G) \in L^2_+(\mathbb{R}; \mathbb{C}^{M \times d}),\\
& \mathbf{T}^{(\mathbf{l})}_{V} : F \in H^1_+(\mathbb{R}; \mathbb{C}^{d \times M}) \mapsto \mathbf{T}^{(\mathbf{l})}_{V}(F) = \Pi_{\geq 0}(FV) \in L^2_+(\mathbb{R}; \mathbb{C}^{d \times N}).\\
\end {split}
\end{equation}If $F=\sum_{j=1}^d \sum_{m=1}^M F_{jm}\mathbb{E}^{(dM)}_{jM} \in  L^2_+(\mathbb{R}; \mathbb{C}^{d \times N})$ and $G=\sum_{n=1}^N \sum_{k=1}^d G_{nk}\mathbb{E}^{(Nd)}_{nk} \in  L^2_+(\mathbb{R}; \mathbb{C}^{N \times d})$, then
\begin{equation}\label{relTrTltoT}
\mathbf{T}^{(\mathbf{r})}_{V}(G) = \sum_{k=1}^M  \sum_{j=1}^d (\sum_{n=1}^N T_{V_{kn}}(G_{nj}))\mathbb{E}^{(Md)}_{kj}; \quad \mathbf{T}^{(\mathbf{l})}_{V}(F) = \sum_{k=1}^d  \sum_{j=1}^N (\sum_{m=1}^M T_{V_{mj}}(F_{km}))\mathbb{E}^{(dN)}_{kj}. 
\end{equation}If $V \in L^{\infty}(\mathbb{R}; \mathbb{C}^{M \times N})$, then $\mathbf{T}^{(\mathbf{r})}_{V}:  L^2_+(\mathbb{R}; \mathbb{C}^{N \times d}) \to L^2_+(\mathbb{R}; \mathbb{C}^{M \times d})$ and $\mathbf{T}^{(\mathbf{l})}_{V}: L^2_+(\mathbb{R}; \mathbb{C}^{d \times M}) \to L^2_+(\mathbb{R}; \mathbb{C}^{d \times N})$ are both bounded. If $ G \in L^2_+(\mathbb{R}; \mathbb{C}^{N \times d})$,   $A \in L^2_+(\mathbb{R}; \mathbb{C}^{M \times d})$, $F \in L^2_+(\mathbb{R}; \mathbb{C}^{d \times M})$, $B \in L^2_+(\mathbb{R}; \mathbb{C}^{d \times N})$,  
\begin{equation}\label{<TrG,A>}
\small
\langle \mathbf{T}^{(\mathbf{r})}_{V}(G), A\rangle_{L^2_+(\mathbb{R}; \mathbb{C}^{M \times d})} = \langle G, \mathbf{T}^{(\mathbf{r})}_{V^*}(A)\rangle_{L^2_+(\mathbb{R}; \mathbb{C}^{N \times d})}; \quad \langle \mathbf{T}^{(\mathbf{l})}_{V}(F), B\rangle_{L^2_+(\mathbb{R}; \mathbb{C}^{d \times N})} = \langle F, \mathbf{T}^{(\mathbf{l})}_{V^*}(B)\rangle_{L^2_+(\mathbb{R}; \mathbb{C}^{d \times M})}. 
\end{equation}Set $M=N$. If $V \in L^{\infty}(\mathbb{R}; \mathbb{C}^{N \times N})$, then \eqref{<TrG,A>} implies that $\mathbf{T}^{(\mathbf{r})}_{V^*} = (\mathbf{T}^{(\mathbf{r})}_{V})^* \in \mathcal{B}(L^2_+(\mathbb{R}; \mathbb{C}^{N \times d}))$ and $\mathbf{T}^{(\mathbf{l})}_{V^*} = (\mathbf{T}^{(\mathbf{l})}_{V})^* \in \mathcal{B}(L^2_+(\mathbb{R}; \mathbb{C}^{d \times N}))$. The right Toeplitz operators are conjugate to the left Toeplitz operators via the transpose transform \eqref{Transpose},
\begin{equation}\label{TconjTrtoTl}
\mathbf{T}^{(\mathbf{r})}_{V} = \mathfrak{T}  \mathbf{T}^{(\mathbf{l})}_{V^{\mathrm{T}}}  \mathfrak{T}, \quad \mathbf{T}^{(\mathbf{l})}_{V} = \mathfrak{T}  \mathbf{T}^{(\mathbf{r})}_{V^{\mathrm{T}}}  \mathfrak{T}.
\end{equation}
\begin{lem}
Given $M,N,d \in \mathbb{N}_+$, if $V \in L^2_-(\mathbb{R}; \mathbb{C}^{M \times N})$, $\forall (A,B)\in \mathbb{C}^{d \times M} \times \mathbb{C}^{N \times d}$,  we have
\begin{equation}\label{TVchiAto0}
\lim_{\epsilon \to 0^+} \left(\|\mathbf{T}^{(\mathbf{l})}_{V}  ( \chi_{\epsilon} A)\|_{L^2_+(\mathbb{R}; \mathbb{C}^{d \times N})} + \|\mathbf{T}^{(\mathbf{r})}_{V}  (\chi_{\epsilon}B)\|_{L^2_+(\mathbb{R}; \mathbb{C}^{M \times d})}  \right) =0.
\end{equation}
\end{lem}
\begin{proof}
Since $AV \in L^2_-(\mathbb{R}; \mathbb{C}^{d \times N})$,  \eqref{CVdominesca} yields that $\mathbf{T}^{(\mathbf{l})}_{V}  ( \chi_{\epsilon} A) \to \Pi_{\geq 0}(AV)=0_{d \times N}$ in $L^2_+(\mathbb{R}; \mathbb{C}^{d \times N})$, as $\epsilon \to 0^+$. The second limit holds thanks to $\|\mathbf{T}^{(\mathbf{r})}_{V}  (\chi_{\epsilon}B)\|_{L^2_+(\mathbb{R}; \mathbb{C}^{M \times d})} = \|\mathbf{T}^{(\mathbf{l})}_{V^{\mathrm{T}}}  (\chi_{\epsilon}B^{\mathrm{T}})\|_{L^2_+(\mathbb{R}; \mathbb{C}^{d \times M})}$.
\end{proof}

\noindent If $U \in  L^{\infty}(\mathbb{R}; \mathbb{C}^{M \times N})$, $V \in  L^{\infty}(\mathbb{R}; \mathbb{C}^{P \times N})$ and $W \in L^{\infty} (\mathbb{R}; \mathbb{C}^{P\times Q})$ for some   $M,N, P, Q\in\mathbb{N}_+$, the matrices $U(x) \in \mathbb{C}^{M \times N}$ and  $V(x) \in \mathbb{C}^{P \times N}$ have the same number of columns,  $\forall x\in \mathbb{R}$, then the $(\mathbf{rr})$-double Toeplitz operators are defined as
\begin{equation}\label{TrTrdef}
\begin{split}
  \mathbf{T}^{(\mathbf{r})}_{U} \mathbf{T}^{(\mathbf{r})}_{V^*}: G_1 \in  L^2_+(\mathbb{R}; \mathbb{C}^{P \times d}) \mapsto   \mathbf{T}^{(\mathbf{r})}_{U} \mathbf{T}^{(\mathbf{r})}_{V^*}(G_1) = \Pi_{\geq 0}\left(U \left(\Pi_{\geq 0 }(V^* G_1 )\right) \right)\in L^2_+(\mathbb{R}; \mathbb{C}^{M \times d}); \\
  \mathbf{T}^{(\mathbf{r})}_{V} \mathbf{T}^{(\mathbf{r})}_{U^*}: G_2 \in  L^2_+(\mathbb{R}; \mathbb{C}^{M \times d}) \mapsto   \mathbf{T}^{(\mathbf{r})}_{V} \mathbf{T}^{(\mathbf{r})}_{U^*}(G_2) = \Pi_{\geq 0}\left(V \left(\Pi_{\geq 0 }(U^* G_2)\right) \right)\in L^2_+(\mathbb{R}; \mathbb{C}^{P \times d}),
\end{split}
\end{equation}for any $d \in \mathbb{N}_+$. The matrices $V(x) \in \mathbb{C}^{P \times N}$ and $W(x) \in \mathbb{C}^{P \times Q}$ have the same number of rows, the $(\mathbf{ll})$-double Toeplitz operators are given by
\begin{equation}\label{TlTldef}
\begin{split}
&\mathbf{T}^{(\mathbf{l})}_{V} \mathbf{T}^{(\mathbf{l})}_{W^*} : F_1 \in L^2_+(\mathbb{R}; \mathbb{C}^{d \times Q}) \mapsto \mathbf{T}^{(\mathbf{l})}_{V} \mathbf{T}^{(\mathbf{l})}_{W^*}(F_1) = \Pi_{\geq 0}\left(\left(\Pi_{\geq 0 }(F_1 W^* )\right) V \right) \in L^2_+(\mathbb{R}; \mathbb{C}^{d \times N});\\
&\mathbf{T}^{(\mathbf{l})}_{W} \mathbf{T}^{(\mathbf{l})}_{V^*} : F_2 \in L^2_+(\mathbb{R}; \mathbb{C}^{d \times N}) \mapsto \mathbf{T}^{(\mathbf{l})}_{W} \mathbf{T}^{(\mathbf{l})}_{V^*}(F_2) = \Pi_{\geq 0}\left(\left(\Pi_{\geq 0 }(F_2 V^* )\right) W \right) \in L^2_+(\mathbb{R}; \mathbb{C}^{d \times Q}),
\end{split}
\end{equation}If $G_1 \in L^2_+(\mathbb{R}; \mathbb{C}^{P \times d}) $, $G_2 \in L^2_+(\mathbb{R}; \mathbb{C}^{M \times d})$, $F_1 \in L^2_+(\mathbb{R}; \mathbb{C}^{d \times Q})$ and $F_2 \in L^2_+(\mathbb{R}; \mathbb{C}^{d \times N})$, we have 
\begin{equation}\label{TTadj}
\small
\begin{cases}
\langle \mathbf{T}^{(\mathbf{r})}_{U} \mathbf{T}^{(\mathbf{r})}_{V^*} (G_1), G_2 \rangle_{ L^2_+(\mathbb{R}; \mathbb{C}^{M \times d})} = \langle \mathbf{T}^{(\mathbf{r})}_{V^*}(G_1) , \mathbf{T}^{(\mathbf{r})}_{U^*} (G_2) \rangle_{ L^2_+(\mathbb{R}; \mathbb{C}^{N \times d})} = \langle  G_1 , \mathbf{T}^{(\mathbf{r})}_{V} \mathbf{T}^{(\mathbf{r})}_{U^*} (G_2) \rangle_{ L^2_+(\mathbb{R}; \mathbb{C}^{P \times d})};\\
\langle \mathbf{T}^{(\mathbf{l})}_{V} \mathbf{T}^{(\mathbf{l})}_{W^*} (F_1), F_2 \rangle_{ L^2_+(\mathbb{R}; \mathbb{C}^{d \times N})} = \langle \mathbf{T}^{(\mathbf{l})}_{W^*}(F_1) , \mathbf{T}^{(\mathbf{l})}_{V^*} (F_2) \rangle_{ L^2_+(\mathbb{R}; \mathbb{C}^{d \times P})} = \langle  F_1 , \mathbf{T}^{(\mathbf{l})}_{V} \mathbf{T}^{(\mathbf{l})}_{W^*} (F_2) \rangle_{ L^2_+(\mathbb{R}; \mathbb{C}^{d \times Q})},\\
\end{cases}
\end{equation}thanks to formula \eqref{<TrG,A>}. Indeed, $(\mathbf{rr})$-double Toeplitz operators are conjugated to $(\mathbf{ll})$-double Hankel operators via the transpose transform \eqref{Transpose},  
\begin{equation}\label{TconjTT}
\mathbf{T}^{(\mathbf{r})}_{U} \mathbf{T}^{(\mathbf{r})}_{V^*} = \mathfrak{T} \mathbf{T}^{(\mathbf{l})}_{U^{\mathrm{T}}} \mathbf{T}^{(\mathbf{l})}_{\overline{V}} \mathfrak{T};  \quad \mathbf{T}^{(\mathbf{l})}_{V} \mathbf{T}^{(\mathbf{l})}_{W^*} = \mathfrak{T} \mathbf{T}^{(\mathbf{r})}_{V^{\mathrm{T}}} \mathbf{T}^{(\mathbf{r})}_{\overline{W}} \mathfrak{T}.
\end{equation}If $M=P$, then $\mathbf{T}^{(\mathbf{r})}_{U} \mathbf{T}^{(\mathbf{r})}_{V^*}$ is a $\mathbb{C}$-linear bounded operator on $L^2_+(\mathbb{R}; \mathbb{C}^{M \times d})$ and $\mathbf{T}^{(\mathbf{r})}_{V} \mathbf{T}^{(\mathbf{r})}_{U^*} = \left( \mathbf{T}^{(\mathbf{r})}_{U} \mathbf{T}^{(\mathbf{r})}_{V^*} \right)^*$. So  $\mathbf{T}^{(\mathbf{r})}_{U} \mathbf{T}^{(\mathbf{r})}_{U^*}\geq 0$ is a $\mathbb{C}$-linear bounded positive self-adjoint operator on $L^2_+(\mathbb{R}; \mathbb{C}^{M \times d})$. Similarly, if $N=Q$, then $\mathbf{T}^{(\mathbf{l})}_{V} \mathbf{T}^{(\mathbf{l})}_{W^*}$ is a $\mathbb{C}$-linear bounded operator on $L^2_+(\mathbb{R}; \mathbb{C}^{d \times N})$ and $\mathbf{T}^{(\mathbf{l})}_{W} \mathbf{T}^{(\mathbf{l})}_{V^*} = \left( \mathbf{T}^{(\mathbf{l})}_{V} \mathbf{T}^{(\mathbf{l})}_{W^*} \right)^*$. So $\mathbf{T}^{(\mathbf{l})}_{V} \mathbf{T}^{(\mathbf{l})}_{V^*}\geq 0$ is a $\mathbb{C}$-linear bounded positive self-adjoint operator on $L^2_+(\mathbb{R}; \mathbb{C}^{d \times N})$. Every double Hankel operator can be expressed in terms of the Toeplitz operators thanks to the next lemma.
\begin{lem}\label{2H2Tlemma}
Given  $M,N, P, Q\in\mathbb{N}_+$,  if $U \in  H^{\frac{1}{2}}_+ \bigcup L^{\infty} (\mathbb{R}; \mathbb{C}^{M \times N})$, $V \in  H^{\frac{1}{2}}_+ \bigcup L^{\infty} (\mathbb{R}; \mathbb{C}^{P \times N})$ and $W \in  H^{\frac{1}{2}}_+ \bigcup L^{\infty} (\mathbb{R}; \mathbb{C}^{P\times Q})$, then we have
\begin{equation}\label{rel2HT-TTfor}
\mathbf{H}^{(\mathbf{r})}_{U} \mathbf{H}^{(\mathbf{l})}_{V}  = \mathbf{T}^{(\mathbf{r})}_{U V^*}  - \mathbf{T}^{(\mathbf{r})}_{U}\mathbf{T}^{(\mathbf{r})}_{V^*}, \quad \mathbf{H}^{(\mathbf{l})}_{V} \mathbf{H}^{(\mathbf{r})}_{W} = \mathbf{T}^{(\mathbf{l})}_{W^* V}  - \mathbf{T}^{(\mathbf{l})}_{V}\mathbf{T}^{(\mathbf{l})}_{W^*}.
\end{equation}
\end{lem}
\begin{proof}
If $G  \in H^1_+(\mathbb{R}; \mathbb{C}^{P \times d})$, $F \in  H^1_+(\mathbb{T}; \mathbb{C}^{d \times Q})$ for some $d \in \mathbb{N}_+$, formula \eqref{Pi<0FR} yields that
\begin{equation*}
\begin{split}
 \mathbf{H}^{(\mathbf{r})}_{U} \mathbf{H}^{(\mathbf{l})}_{V} (G) = \Pi_{\geq 0} \left(U \Pi_{<0}(V^* G) \right) = \Pi_{\geq 0} \left(UV^* G- U \Pi_{\geq 0}(V^* G) \right)  = (\mathbf{T}^{(\mathbf{r})}_{U V^*}  - \mathbf{T}^{(\mathbf{r})}_{U}\mathbf{T}^{(\mathbf{r})}_{V^*}) (G).
\end{split}
\end{equation*}The second formula can be deduced from the first one and the conjugate relations \eqref{TconjHH}, \eqref{TconjTrtoTl}, \eqref{TconjTT}.
\end{proof}
 
\noindent Inspired from Sun \cite{SUNInterCMSDNLS}, we recall the commutator formula between the infinitesimal generator $\mathbf{G}$ and the Toeplitz operator in the next lemma.
\begin{lem}\label{Lemm[T,G]}
Given $M,N \in \mathbb{N}_+$, if $V \in L^2 \bigcap L^{\infty}(\mathbb{R}; \mathbb{C}^{M \times N})$,  then $\mathbf{T}_{V}^{(\mathbf{r})} \left(\mathrm{Dom}(\mathbf{G})^{N\times d} \right) \subset \mathrm{Dom}(\mathbf{G})^{M\times d}$ and $\mathbf{T}_{V}^{(\mathbf{l})} \left(\mathrm{Dom}(\mathbf{G})^{d\times M} \right) \subset \mathrm{Dom}(\mathbf{G})^{d\times N}$, $\forall d\in \mathbb{N}_+$. Moreover, the following identities hold,
\begin{equation}\label{for[T,G]}
[\mathbf{G}, \mathbf{T}_{V}^{(\mathbf{r})}] (\Phi) = \tfrac{i\left(\Pi_{\geq 0}V \right)\hat{\Phi}(0^+)}{2\pi} \in L^2_+(\mathbb{R}; \mathbb{C}^{M \times d}), \quad [\mathbf{G}, \mathbf{T}_{V}^{(\mathbf{l})}] (\Psi) = \tfrac{i\hat{\Psi}(0^+) \Pi_{\geq 0}V  }{2\pi} \in L^2_+(\mathbb{R}; \mathbb{C}^{d \times N}),
\end{equation}$\forall \Phi \in \mathrm{Dom}(\mathbf{G})^{N\times d}$ and $\forall \Psi \in \mathrm{Dom}(\mathbf{G})^{d\times M}$.
\end{lem}
\begin{proof}
The first formula is proved by lemma 4.1 of Sun  \cite{SUNInterCMSDNLS}. The second formula can be deduced by the first formula and the conjugate relations  \eqref{TF=FT}, \eqref{TG=GT}, \eqref{TconjTrtoTl} and \eqref{TconjTT}.
\end{proof}
\noindent The next lemma on the commutator formula between the infinitesimal generator $\mathbf{G}$ and double Toeplitz operators given by \eqref{TrTrdef}, \eqref{TlTldef} is a direct corollary of lemma $\ref{Lemm[T,G]}$.
\begin{lem}\label{Lemm[TT,G]}
If $U \in L^2_+ \bigcap L^{\infty}(\mathbb{R}; \mathbb{C}^{M \times N})$, $V \in L^2_+ \bigcap L^{\infty}(\mathbb{R}; \mathbb{C}^{P \times N})$ and $W \in L^2_+ \bigcap L^{\infty} (\mathbb{R}; \mathbb{C}^{P\times Q})$ for some positive integers $M,N, P, Q\in\mathbb{N}_+$, then $\forall d\in \mathbb{N}_+$, we have $\mathbf{H}_{U}^{(\mathbf{r})}\mathbf{H}_{V}^{(\mathbf{l})}\left(\mathrm{Dom}(\mathbf{G})^{P\times d} \right) \subset \mathrm{Dom}(\mathbf{G})^{M\times d}$ and $\mathbf{H}_{V}^{(\mathbf{l})}\mathbf{H}_{W}^{(\mathbf{r})}\left(\mathrm{Dom}(\mathbf{G})^{d\times Q} \right) \subset \mathrm{Dom}(\mathbf{G})^{d\times N}$.  Moreover, the following identities hold,
\begin{equation}\label{for[TT,G]}
\small
\begin{split}
& [\mathbf{G}, \mathbf{T}_{UV^*}^{(\mathbf{r})}] (\Phi) = [\mathbf{G}, \mathbf{H}_{U}^{(\mathbf{r})}\mathbf{H}_{V}^{(\mathbf{l})}] (\Phi) + \sum_{k=1}^N \sum_{j=1}^d \lim_{\epsilon \to 0^+} \langle \Phi, \chi_{\epsilon}V \mathbb{E}_{kj}^{(Nd)} \rangle_{L^2_+(\mathbb{R}; \mathbb{C}^{P \times d})} \tfrac{i U \mathbb{E}_{kj}^{(Nd)}}{2\pi}\in L^2_+(\mathbb{R}; \mathbb{C}^{M \times d});\\
& [\mathbf{G}, \mathbf{T}_{W^* V}^{(\mathbf{l})}] (\Psi) = [\mathbf{G}, \mathbf{H}_{V}^{(\mathbf{l})} \mathbf{H}_{W}^{(\mathbf{r})}] (\Psi) + \sum_{k=1}^d \sum_{j=1}^P \lim_{\epsilon \to 0^+} \langle \Psi, \chi_{\epsilon}  \mathbb{E}_{kj}^{(dP)}W \rangle_{L^2_+(\mathbb{R}; \mathbb{C}^{d \times Q})} \tfrac{i  \mathbb{E}_{kj}^{(dP)}V}{2\pi}\in L^2_+(\mathbb{R}; \mathbb{C}^{d \times N}),
\end{split}
\end{equation}$\forall \Phi \in \mathrm{Dom}(\mathbf{G})^{P\times d}$ and $\forall \Psi \in \mathrm{Dom}(\mathbf{G})^{d\times Q}$.
\end{lem}
\begin{proof}
We have $\mathbf{T}_{UV^*}^{(\mathbf{r})} \left( \mathrm{Dom}(\mathbf{G})^{P\times d}\right) \subset \mathrm{Dom}(\mathbf{G})^{M\times d}$ and $\mathbf{T}_{U }^{(\mathbf{r})} \mathbf{T}_{V^*}^{(\mathbf{r})}  \left( \mathrm{Dom}(\mathbf{G})^{P\times d}\right) \subset \mathrm{Dom}(\mathbf{G})^{M\times d}$ by lemma $\ref{Lemm[T,G]}$. Then \eqref{TconjHH} and \eqref{rel2HT-TTfor} yield  that $\mathbf{H}^{(\mathbf{r})}_{U} \mathbf{H}^{(\mathbf{l})}_{V} = \mathfrak{T} \mathbf{H}^{(\mathbf{l})}_{U^{\mathrm{T}}} \mathbf{H}^{(\mathbf{r})}_{V^{\mathrm{T}}} \mathfrak{T} : \mathrm{Dom}(\mathbf{G})^{P\times d}  \to \mathrm{Dom}(\mathbf{G})^{M\times d}$. Since  $V^* \in L^2_- \bigcap L^{\infty}(\mathbb{R}; \mathbb{C}^{N \times P})$, then $[\mathbf{G}, \mathbf{T}_{V^*}^{(\mathbf{r})}] (\Phi) = \tfrac{i}{2\pi}\left(\Pi_{\geq 0}V^* \right)\hat{\Phi}(0^+)=0_{N \times d}$, $\forall \Phi \in \mathrm{Dom}(\mathbf{G})^{P\times d}$. So
\begin{equation}\label{[G,TUTVstar]1}
[\mathbf{G}, \mathbf{T}_{U}^{(\mathbf{r})} \mathbf{T}_{V^*}^{(\mathbf{r})}] (\Phi) = [\mathbf{G}, \mathbf{T}_{U}^{(\mathbf{r})} ] \mathbf{T}_{V^*}^{(\mathbf{r})}(\Phi) = \tfrac{iU}{2\pi} \hat{F}(0^+) \quad \mathrm{with} \quad F= \mathbf{T}_{V^*}^{(\mathbf{r})}(\Phi) \in \mathrm{Dom}(\mathbf{G})^{N\times d}.
\end{equation}Lemma $\ref{ABL2+forR}$ implies that  $\chi_{\epsilon}V  \mathbb{E}_{kj}^{(Nd)} = \mathbf{T}_{V}^{(\mathbf{r})}(\chi_{\epsilon}  \mathbb{E}_{kj}^{(Nd)}) \in L^2_+(\mathbb{R}; \mathbb{C}^{P \times d})$, then we have
\begin{equation}\label{hatF0+Phi}
\small
\hat{F}(0^+) = \sum_{k=1}^N \sum_{j=1}^d \lim_{\epsilon \to 0^+} \langle \mathbf{T}_{V^*}^{(\mathbf{r})}(\Phi), \chi_{\epsilon}  \mathbb{E}_{kj}^{(Nd)} \rangle_{L^2_+(\mathbb{R}; \mathbb{C}^{N \times d})}  \mathbb{E}_{kj}^{(Nd)} =\sum_{k=1}^N \sum_{j=1}^d \lim_{\epsilon \to 0^+} \langle \Phi, \chi_{\epsilon}V \mathbb{E}_{kj}^{(Nd)} \rangle_{L^2_+(\mathbb{R}; \mathbb{C}^{P \times d})}  \mathbb{E}_{kj}^{(Nd)},
\end{equation}by using  \eqref{HATF0POsitiv} and \eqref{<TrG,A>}. Then the first formula of \eqref{for[TT,G]} is obtained by plugging formulas \eqref{rel2HT-TTfor} and \eqref{hatF0+Phi} into \eqref{[G,TUTVstar]1}. The second formula of \eqref{for[TT,G]} is equivalent to the first one through the conjugate relations \eqref{TF=FT}, \eqref{TG=GT}, \eqref{TconjHH}, \eqref{TconjTrtoTl} and \eqref{TconjTT}.
\end{proof}

\subsection{Proof of Lax pair theorem $\ref{LaxPairThmR}$} 
\begin{lem}\label{HUVWKUVWLem}
Given $s> \frac{1}{2}$ and $M,N,P,Q \in \mathbb{N}_+$, if $U \in H^s_+(\mathbb{R}; \mathbb{C}^{M \times N})$, $V \in H^s_+(\mathbb{R}; \mathbb{C}^{P \times N})$ and $W \in H^s_+(\mathbb{R}; \mathbb{C}^{P\times Q})$, then $\forall d \in \mathbb{N}_+$,  the following identities hold,
\begin{equation}\label{HUVWfor}
\begin{split}
\mathbf{H}^{(\mathbf{r})}_{\Pi_{\geq 0}(UV^*W)} = &\mathbf{T}^{(\mathbf{r})}_{UV^*}\mathbf{H}^{(\mathbf{r})}_W + \mathbf{H}^{(\mathbf{r})}_U \mathbf{T}^{(\mathbf{l})}_{W^* V} - \mathbf{H}^{(\mathbf{r})}_U \mathbf{H}^{(\mathbf{l})}_V \mathbf{H}^{(\mathbf{r})}_W: L^2_+(\mathbb{R}; \mathbb{C}^{d \times Q}) \to L^2_+(\mathbb{R}; \mathbb{C}^{M \times d});\\
\mathbf{H}^{(\mathbf{l})}_{\Pi_{\geq 0}(UV^*W)} = & \mathbf{T}^{(\mathbf{l})}_{V^* W}\mathbf{H}^{(\mathbf{l})}_U + \mathbf{H}^{(\mathbf{l})}_W \mathbf{T}^{(\mathbf{r})}_{V U^* } - \mathbf{H}^{(\mathbf{l})}_W \mathbf{H}^{(\mathbf{r})}_V \mathbf{H}^{(\mathbf{l})}_U: L^2_+(\mathbb{R}; \mathbb{C}^{M \times d}) \to L^2_+(\mathbb{R}; \mathbb{C}^{d \times Q}).\\ 
\end{split}
\end{equation}
\end{lem}
\noindent The proof of lemma $\ref{HUVWKUVWLem}$ is slightly different from lemma 3.11 of Sun \cite{SUNMSzego} due to the $0$-mode. 
\begin{proof}
Since $UV^*W \in H^1(\mathbb{R}; \mathbb{C}^{M \times Q})$,  $\forall F \in L^2_+(\mathbb{R}; \mathbb{C}^{d \times Q})$, we have $\Pi_{<0}(UV^* W)F^* \in L^2_-(\mathbb{R}; \mathbb{C}^{M \times d})$ by lemma $\ref{ABL2+forR}$. Formula \eqref{Pi<0FR} yields that $\Pi_{<0}(WF^*) = \left(\Pi_{\geq 0}(F W^*) \right)^* \in L^2_-(\mathbb{R}; \mathbb{C}^{P \times d})$. Then
\begin{equation}\label{Hrpiuvw1}
\mathbf{H}^{(\mathbf{r})}_{\Pi_{\geq 0}(UV^*W)} (F) = \Pi_{\geq 0}(UV^* W F^*) =  \mathbf{T}^{(\mathbf{r})}_{UV^*}\mathbf{H}^{(\mathbf{r})}_W(F) + \mathbf{H}^{(\mathbf{r})}_{U} \left( \Pi_{\geq 0}(FW^*)V\right) \in  L^2_+ (\mathbb{R}; \mathbb{C}^{M \times d}).
\end{equation}Using lemma $\ref{ABL2+forR}$ again, we have $ \Pi_{\geq 0}(FW^*)V = \mathbf{T}^{(\mathbf{l})}_{V}\mathbf{T}^{(\mathbf{l})}_{W^*}(F)\in L^2_+ (\mathbb{R}; \mathbb{C}^{d \times N})$. The first formula of \eqref{HUVWfor} is  obtained by plugging \eqref{rel2HT-TTfor} into \eqref{Hrpiuvw1}. The second formula of \eqref{HUVWfor} is deduced from the first one through the conjugate relations \eqref{TF=FT}, \eqref{TconjHrtoHl} \eqref{TconjHH}, \eqref{TconjTrtoTl} and \eqref{TconjTT}. Precisely, we have
\begin{equation}\label{TconjHlPiUVstarW}
\mathbf{H}^{(\mathbf{l})}_{\Pi_{\geq 0}(UV^*W)} = \mathfrak{T} \mathbf{H}^{(\mathbf{r})}_{\Pi_{\geq 0}(W^{\mathrm{T}} (V^{\mathrm{T}})^* U^{\mathrm{T}})}\mathfrak{T} =\mathfrak{T}  \left( \mathbf{T}^{(\mathbf{r})}_{W^{\mathrm{T}}  \overline{V}} \mathbf{H}_{U^{\mathrm{T}}}^{(\mathbf{r})} +  \mathbf{H}_{W^{\mathrm{T}}}^{(\mathbf{r})} \mathbf{T}^{(\mathbf{l})}_{\overline{U} V^{\mathrm{T}}  } - \mathbf{H}_{W^{\mathrm{T}}}^{(\mathbf{r})}\mathbf{H}_{V^{\mathrm{T}}}^{(\mathbf{l})}\mathbf{H}_{U^{\mathrm{T}}}^{(\mathbf{r})} \right)\mathfrak{T}.
\end{equation}
\end{proof} 
\noindent The proof of   proposition $\ref{LaxPairThmR}$ is similar to  theorem 1.3 of Sun \cite{SUNMSzego}. 
\begin{proof}[Proof of proposition $\ref{LaxPairThmR}$]
Given  $s>\frac{1}{2}$, set $V=W=U\in H^s_+(\mathbb{R};  \mathbb{C}^{M\times N})$ in formula \eqref{HUVWfor}. Then
\begin{equation}
\begin{split}
& \mathbf{H}^{(\mathbf{r})}_{U}\mathbf{H}^{(\mathbf{l})}_{\Pi_{\geq 0}(U U^* U)} - \mathbf{H}^{(\mathbf{r})}_{\Pi_{\geq 0}(U U^* U)} \mathbf{H}^{(\mathbf{l})}_{U} = \left[\mathbf{H}^{(\mathbf{r})}_{U}\mathbf{H}^{(\mathbf{l})}_{U}, \; \mathbf{T}^{(\mathbf{r})}_{U U^*}\right] \in \mathcal{S}_1(L^2_+(\mathbb{R};  \mathbb{C}^{M\times d})); \\
& \mathbf{H}^{(\mathbf{l})}_{U}\mathbf{H}^{(\mathbf{r})}_{\Pi_{\geq 0}(U U^* U)} - \mathbf{H}^{(\mathbf{l})}_{\Pi_{\geq 0}(U U^* U)} \mathbf{H}^{(\mathbf{r})}_{U} = \left[\mathbf{H}^{(\mathbf{l})}_{U}\mathbf{H}^{(\mathbf{r})}_{U}, \; \mathbf{T}^{(\mathbf{l})}_{U^* U}\right] \in \mathcal{S}_1(L^2_+(\mathbb{R};  \mathbb{C}^{d\times N})).
\end{split}
\end{equation}where $\mathcal{S}_1(L^2_+(\mathbb{R};  \mathbb{C}^{M\times d}))$ denotes the space of trace-class operators on $L^2_+(\mathbb{R};  \mathbb{C}^{M\times d})$.
\end{proof}

\begin{rem}
Thanks to formula \eqref{rel2HT-TTfor}, $(\mathbf{H}^{(\mathbf{r})}_{U}\mathbf{H}^{(\mathbf{l})}_{U}, -i \mathbf{T}^{(\mathbf{r})}_{U} \mathbf{T}^{(\mathbf{r})}_{ U^*} )$ and 
 $(\mathbf{H}^{(\mathbf{l})}_{U}\mathbf{H}^{(\mathbf{r})}_{U}, -i \mathbf{T}^{(\mathbf{l})}_{U} \mathbf{T}^{(\mathbf{l})}_{U^*})$ are also Lax pairs of the matrix Szeg\H{o} equation \eqref{MSzego}. 
\end{rem}

\section{The explicit formula}\label{SectExplForR}
\noindent This section is devoted to establishing the explicit formula \eqref{expforMSzegoR} in theorem $\ref{MSzegoRExp}$. The proof is based on the Lax pair structure in proposition $\ref{LaxPairThmR}$. Due to the unitary equivalence \eqref{TconjHH} between the two lax operators $\mathbf{H}^{(\mathbf{l})}_{U}\mathbf{H}^{(\mathbf{r})}_{U}$ and $\mathbf{H}^{(\mathbf{r})}_{U^{\mathrm{T}}}\mathbf{H}^{(\mathbf{l})}_{U^{\mathrm{T}}}$.   Let $\mathbf{L}_U:=\mathbf{H}^{(\mathbf{l})}_{U}\mathbf{H}^{(\mathbf{r})}_{U} \in \mathcal{S}_1(L^2_+(\mathbb{R};  \mathbb{C}^{d\times N}))$ and $\mathbf{A}_U := -i \mathbf{T}^{(\mathbf{l})}_{U^* U}$ denote  the canonical Lax pair of the matrix Szeg\H{o} equation \eqref{MSzego} on the line. Then  \eqref{TwoHeiLaxMSzegoR} reads as
\begin{equation}
\partial_t \mathbf{L}_{U(t)} = [\mathbf{A}_{U(t)}, \mathbf{L}_{U(t)}], 
\end{equation}Let $\mathbb{X}_{MN}:= \mathcal{B}_{\mathbb{C}}(L^2_+(\mathbb{R}; \mathbb{C}^{M \times N}))$, $\forall M,N \in \mathbb{N}_+$. Assume that $U \in C^{\infty}(\mathbb{R}; H^s_+(\mathbb{R};  \mathbb{C}^{M\times N}))$ solves \eqref{MSzego}, $\forall d \in \mathbb{N}_+$,  the mapping $\mathcal{A}^{(\mathbf{l})}: t \in \mathbb{R} \mapsto \mathcal{A}^{(\mathbf{l})}(t)  \in   \mathcal{B}_{\mathbb{C}}(\mathbb{X}_{dN})$ is continuous, where
\begin{equation*}
\mathcal{A}^{(\mathbf{l})}(t): \mathscr{W}\in \mathbb{X}_{dN} \mapsto \mathcal{A}^{(\mathbf{l})}(t)(\mathscr{W}) = \mathbf{A}_{U(t)} \mathscr{W}=-i \mathbf{T}^{(\mathbf{l})}_{ U(t)^* U(t)}\mathscr{W} \in \mathbb{X}_{dN}.
\end{equation*}Proposition $\ref{CauchyThmODE}$ yields that there exists a unique function $\mathscr{W}\in C^1(\mathbb{R}; \mathcal{B}_{\mathbb{C}}(L^2_+(\mathbb{R};\mathbb{C}^{d \times N})))$ solving 
\begin{equation}\label{WODER}
\begin{split}
 \tfrac{\mathrm{d}}{\mathrm{d}t}\mathscr{W}(t) = \mathbf{A}_{U(t)} \mathscr{W}(t)  = -i \mathbf{T}^{(\mathbf{l})}_{U(t)^* U(t)}\mathscr{W}(t), \quad \mathscr{W}(0)= \mathrm{id}_{L^2_+(\mathbb{R};\mathbb{C}^{d \times N})}.
\end{split}
\end{equation}Let $\mathscr{W}(t)^* \in \mathbb{X}_{dN}$ denotes the $L^2_+(\mathbb{R}; \mathbb{C}^{d \times N})$-adjoint of $\mathscr{W}(t)$. Since $i \mathbf{A}_{U(t)} = \mathbf{T}^{(\mathbf{l})}_{ U(t)^* U(t)}  \in \mathbb{X}_{dN}$ is self-adjoint, then $\tfrac{\mathrm{d}}{\mathrm{d}t} \left( \mathscr{W}(t)^* \mathscr{W}(t)\right)=0_{\mathbb{X}_{dN}}$ and $\mathscr{W}\mathscr{W}^* \in C^1(\mathbb{R}; \mathcal{B}_{\mathbb{C}}(L^2_+(\mathbb{R};\mathbb{C}^{d \times N})))$ solves
\begin{equation*}
\partial_t \mathscr{V}(t) = [\mathbf{A}_{U(t)},  \mathscr{V}(t)], \quad \mathscr{V}(0)=\mathscr{W}(0)\mathscr{W}(0)^*= \mathrm{id}_{L^2_+(\mathbb{R};\mathbb{C}^{d \times N})}.
\end{equation*}Thanks to the uniqueness argument in proposition $\ref{CauchyThmODE}$, we have $\mathscr{W}(t)\mathscr{W}(t)^*=\mathrm{id}_{L^2_+(\mathbb{R};\mathbb{C}^{d \times N})}$, $\forall t \in \mathbb{R}$. So $\mathscr{W}(t)=\left(\mathscr{W}(t)^* \right)^{-1} \in \mathbb{X}_{dN}$ is a unitary operator and $\mathbf{L}_{U(t)}$ is unitarily equivalent to $\mathbf{L}_{U(0)}$, i.e.
\begin{equation}\label{UniEquiLutLu0}
\mathbf{H}^{(\mathbf{l})}_{U(t)}\mathbf{H}^{(\mathbf{r})}_{U(t)} = \mathscr{W}(t)\mathbf{H}^{(\mathbf{l})}_{U(0)}\mathbf{H}^{(\mathbf{r})}_{U(0)} \mathscr{W}(t)^* \in \mathcal{S}_1(L^2_+(\mathbb{R};  \mathbb{C}^{d\times N})).
\end{equation}The next lemma provides some useful estimates for dominated convergence in the conjugation acting algorithm.
\begin{lem}\label{LemWtstar}
Given $M,N,d \in \mathbb{N}_+$ and $s>\tfrac{1}{2}$, if $U \in C^{\infty}(\mathbb{R}; H^s_+(\mathbb{R};  \mathbb{C}^{M\times N}))$ solves equation \eqref{MSzego} with $U(0)=U_0$, $\forall (Q, P) \in  \mathbb{C}^{d \times M} \times \mathbb{C}^{d \times N}$, we have $\mathscr{W}(t)^* \left(Q U(t) \right) = Q U_0 \in H^s_+(\mathbb{R}; \mathbb{C}^{d \times N})$ and
\begin{equation}\label{W(t)starlinea}
\lim_{\epsilon \to 0^+} \left( \|\mathscr{W}(t)^* \left(\chi_{\epsilon} Q U(t) \right) - \chi_{\epsilon} Q U_0 \|_{L^2_+(\mathbb{R}; \mathbb{C}^{d \times N})} + \|\mathscr{W}(t)^* \left(\chi_{\epsilon} P \right) - e^{it \mathbf{H}^{(\mathbf{l})}_{U_0}\mathbf{H}^{(\mathbf{r})}_{U_0} }(\chi_{\epsilon} P) \|_{L^2_+(\mathbb{R}; \mathbb{C}^{d \times N})} \right) = 0,
\end{equation}where $\chi_{\epsilon}:x \in \mathbb{R} \mapsto (1 -i \epsilon x)^{-1} \in \mathbb{C}$ is defined as \eqref{chieps}, $0 <\epsilon <1$.
\end{lem}
\begin{proof}
Formula \eqref{matrixXL^2} yields that $\partial_t (Q U)(t)= -i\Pi_{\geq 0} \left(Q U(t) U(t)^* U(t) \right)=\mathbf{A}_{U(t)}(QU(t))$. Then
\begin{equation*}
\small
\partial_t\left(\mathscr{W}(t)^* (QU(t)) \right) = \partial_t\mathscr{W}(t)^* \left(QU(t) \right) + \mathscr{W}(t)^* (Q \partial_t U(t)) =  \mathscr{W}(t)^* (-\mathbf{A}_{U(t)} + \mathbf{A}_{U(t)})(QU(t))= 0_{d \times N},
\end{equation*}$\forall t \in \mathbb{R}$, by the previous formula and \eqref{WODER}. In particular, if $1 \leq k \leq d$, $1 \leq j \leq M$, we have
\begin{equation}\label{WstarUt}
\mathscr{W}(t)^* \left(\mathbb{E}^{(dM)}_{kj} U(t) \right) = \mathbb{E}^{(dM)}_{kj} U_0 \in H^s_+(\mathbb{R}; \mathbb{C}^{d \times N}), \quad \mathscr{W}(t)^* \left( U(t) \right) = U_0 \in H^s_+(\mathbb{R}; \mathbb{C}^{M \times N}).
\end{equation}Since $\chi_{\epsilon} \in H^{+\infty}_+(\mathbb{R}; \mathbb{C})$, we have $\chi_{\epsilon}Q\mathbf{A}_{U(t)}(U(t))=-i \Pi_{\geq 0}\left(\chi_{\epsilon}Q \Pi_{\geq 0} \left(U(t) U(t)^* U(t) \right) \right)\in H^s_+(\mathbb{R}; \mathbb{C}^{d \times N})$ by using lemma  $\ref{ABL2+forR}$. Let $V(t):= \Pi_{< 0}\left(U(t) U(t)^* U(t) \right) \in H^s_-(\mathbb{R}; \mathbb{C}^{M \times N})$. Then \eqref{Pi<0FR} and \eqref{WODER} yield that
\begin{equation}\label{derWstarChiQU}
\small
\partial_t\left(\mathscr{W}(t)^* (\chi_{\epsilon} Q U(t)) \right) = \mathscr{W}(t)^* \left(\chi_{\epsilon}Q\mathbf{A}_{U(t)}(U(t)) - \mathbf{A}_{U(t)}(\chi_{\epsilon} Q U(t))\right) = i \mathscr{W}(t)^*  \mathbf{T}^{(\mathbf{l})}_{V(t)}(\chi_{\epsilon } Q).
\end{equation}Since $\sup_{0 <\epsilon <1}\|\chi_{\epsilon}\|_{L^{\infty}(\mathbb{R}; \mathbb{C})} = 1$, Sobolev embedding theorem $H^{\frac{1}{3}}(\mathbb{R}) \hookrightarrow L^6(\mathbb{R})$ implies that
\begin{equation}\label{UniBouTVchiQ}
\sup_{0 <\epsilon <1}\sup_{|\tau|\leq |t|}\|\mathbf{T}^{(\mathbf{l})}_{V(\tau)}(\chi_{\epsilon } Q)\|_{L^2_+(\mathbb{R}; \mathbb{C}^{d \times N})} \lesssim \|Q\|_{\mathbb{C}^{d \times M}} \sup_{|\tau|\leq |t|}\|U(\tau)\|^3_{H^{\frac{1}{3}}_+(\mathbb{R}; \mathbb{C}^{M \times N})}.
\end{equation}Integrating \eqref{derWstarChiQU}, we obtain that $\mathscr{W}(t)^* (\chi_{\epsilon} Q U(t)) - \chi_{\epsilon} Q U_0 = i\int_0^t  \mathscr{W}(\tau)^*  \mathbf{T}^{(\mathbf{l})}_{V(\tau)}(\chi_{\epsilon } Q) \mathrm{d}\tau$. Then
\begin{equation*}
\|\mathscr{W}(t)^* (\chi_{\epsilon} Q U(t)) - \chi_{\epsilon} Q U_0 \|_{L^2_+(\mathbb{R}; \mathbb{C}^{d \times N})} \leq \big|\int_0^t    \|\mathbf{T}^{(\mathbf{l})}_{V(\tau)}(\chi_{\epsilon } Q)\|_{L^2_+(\mathbb{R}; \mathbb{C}^{d \times N})} \mathrm{d}\tau \big| \to 0,
\end{equation*}as $\epsilon \to 0^+$, thanks to formulas \eqref{TVchiAto0}, \eqref{UniBouTVchiQ} and Lebesgue's dominated convergence theorem. Formulas \eqref{rel2HT-TTfor}, \eqref{WODER} and \eqref{UniEquiLutLu0} yield that 
\begin{equation}\label{W'starchiP}
 \partial_t\left(\mathscr{W}(t)^* (\chi_{\epsilon} P) \right)  
 =  i\mathbf{H}^{(\mathbf{l})}_{U_0}\mathbf{H}^{(\mathbf{r})}_{U_0}\mathscr{W}(t)^* (\chi_{\epsilon} P) +  i \mathscr{W}(t)^*   \mathbf{T}^{(\mathbf{l})}_{U(t)} \mathbf{T}^{(\mathbf{l})}_{U(t)^*}   (\chi_{\epsilon}P)\in L^2_+(\mathbb{R}; \mathbb{C}^{d \times N}).
\end{equation}Since $U(t)^* \in H^s_-(\mathbb{R}; \mathbb{C}^{N\times M})$, formula \eqref{TVchiAto0} implies that $\lim_{\epsilon \to 0^+} \|\mathbf{T}^{(\mathbf{l})}_{U(t)^*}(\chi_{\epsilon}P) \|_{L^2_+(\mathbb{R}; \mathbb{C}^{d \times M})}=0$ and 
\begin{equation}\label{UnifBouTUchiP}
 \sup_{0 <\epsilon <1}\sup_{|\tau|\leq |t|}\|\mathbf{T}^{(\mathbf{l})}_{U(\tau)^*}(\chi_{\epsilon } P)\|_{L^2_+(\mathbb{R}; \mathbb{C}^{d \times M})}  \leq \|P\|_{\mathbb{C}^{d \times N}} \sup_{|\tau|\leq |t|}\|U(\tau)\|_{L^2_+(\mathbb{R}; \mathbb{C}^{M \times N})}.
\end{equation}Then we have $e^{-it \mathbf{H}^{(\mathbf{l})}_{U_0}\mathbf{H}^{(\mathbf{r})}_{U_0}} \mathscr{W}(t)^*(\chi_{\epsilon} P) - \chi_{\epsilon} P = i\int_0^t e^{-i\tau \mathbf{H}^{(\mathbf{l})}_{U_0}\mathbf{H}^{(\mathbf{r})}_{U_0}} \mathscr{W}(\tau)^* \mathbf{T}^{(\mathbf{l})}_{U(\tau)} \mathbf{T}^{(\mathbf{l})}_{U(\tau)^*}   (\chi_{\epsilon}P) \mathrm{d}\tau$ by integrating \eqref{W'starchiP}. Thanks to \eqref{TVchiAto0}, \eqref{UnifBouTUchiP} and Lebesgue's dominated convergence theorem, 
\begin{equation}
\small
\|\mathscr{W}(t)^*(\chi_{\epsilon} P) - e^{ it \mathbf{H}^{(\mathbf{l})}_{U_0}\mathbf{H}^{(\mathbf{r})}_{U_0}} (\chi_{\epsilon} P)\|_{L^2_+(\mathbb{R}; \mathbb{C}^{d \times N})} \leq  \sup_{|\tau| \leq |t|}\|U(\tau)\|_{L^{\infty}(\mathbb{R})} \big| \int_0^t \|\mathbf{T}^{(\mathbf{l})}_{U(\tau)^*}   (\chi_{\epsilon}P)\|_{L^2_+(\mathbb{R}; \mathbb{C}^{ d \times M})} \mathrm{d}\tau \big| \to 0,
\end{equation}as $\epsilon \to 0^+$.
\end{proof}
\noindent Then we are ready to prove the main theorem $\ref{MSzegoRExp}$.
\begin{proof}[Proof of theorem $\ref{MSzegoRExp}$]
If $\Psi \in  L^2_+(\mathbb{R}; \mathbb{C}^{d \times N})$ for some $d \in \mathbb{N}_+$, we have $\widehat{\Psi U_0^*} \in C_0 \bigcap L^2(\mathbb{R}; \mathbb{C}^{d \times M})$ and
\begin{equation}\label{PsiU0st}
\mathscr{I}(\mathbf{T}^{(\mathbf{l})}_{U_0^*}( \Psi)) = \lim_{\delta \to 0^+}\widehat{\Psi U_0^*} (\delta)= \widehat{\Psi U_0^*} (0) \in \mathbb{C}^{d \times M}.
\end{equation}Let $\mathscr{W} \in C^1(\mathbb{R}; \mathcal{B}_{\mathbb{C}}(L^2_+(\mathbb{R};\mathbb{C}^{d \times N})))$ denote the unique solution of \eqref{WODER}. If $1 \leq k \leq d$, $1 \leq j \leq M$, lemma $\ref{ABL2+forR}$ yields that $\chi_{\epsilon} \mathbb{E}_{kj}^{(dM)}U_0 = \mathbf{T}^{(\mathbf{l})}_{U_0}(\chi_{\epsilon} \mathbb{E}_{kj}^{(dM)}) \in L^2_+(\mathbb{R}; \mathbb{C}^{d \times N})$. For any $\epsilon \in (0, 1)$, 
\begin{equation}\label{WtPsicEU}
\small
\langle \mathscr{W}(t) \Psi, \chi_{\epsilon} \mathbb{E}_{kj}^{(dM)}U(t)\rangle_{L^2_+(\mathbb{R}; \mathbb{C}^{d \times N})} = \langle  \Psi, \chi_{\epsilon} \mathbb{E}_{kj}^{(dM)}U_0\rangle_{L^2_+} + \mathbf{r}_{\epsilon}^{(1)} = \langle \mathbf{T}^{(\mathbf{l})}_{U_0^*}( \Psi), \chi_{\epsilon} \mathbb{E}_{kj}^{(dM)} \rangle_{L^2_+(\mathbb{R}; \mathbb{C}^{d \times M})} + \mathbf{r}_{\epsilon}^{(1)} ,
\end{equation}thanks to formula \eqref{<TrG,A>}, where $\mathbf{r}_{\epsilon}^{(1)} \in \mathbb{C}$ is given by  
\begin{equation}\label{r1term}
\mathbf{r}_{\epsilon}^{(1)} := \langle \Psi, \mathscr{W}(t)^*  (\chi_{\epsilon} \mathbb{E}_{kj}^{(dM)}U(t)  ) - \chi_{\epsilon} \mathbb{E}_{kj}^{(dM)}U_0 \rangle_{L^2_+(\mathbb{R}; \mathbb{C}^{d \times N})} \to 0,
\end{equation}as $\epsilon \to 0^+$, due to Cauchy--Schwarz inequality and \eqref{W(t)starlinea}. Plugging \eqref{r1term} into \eqref{WtPsicEU}, we obtain
\begin{equation}
\lim_{\epsilon \to 0^+}\langle \mathscr{W}(t) \Psi, \chi_{\epsilon} \mathbb{E}_{kj}^{(dM)}U(t)\rangle_{L^2_+(\mathbb{R}; \mathbb{C}^{d \times N})} = \lim_{\epsilon \to 0^+} \langle \mathbf{T}^{(\mathbf{l})}_{U_0^*}( \Psi), \chi_{\epsilon} \mathbb{E}_{kj}^{(dM)} \rangle_{L^2_+(\mathbb{R}; \mathbb{C}^{d \times M})} = \left(\widehat{\Psi U_0^*}(0)\right)_{kj} \in \mathbb{C}.
\end{equation}Formula \eqref{UniEquiLutLu0} yields that $\mathscr{W}(t)^*[\mathbf{G}, \mathbf{H}^{(\mathbf{l})}_{U(t)}\mathbf{H}^{(\mathbf{r})}_{U(t)}]\mathscr{W}(t) = [\mathscr{W}(t)^* \mathbf{G}\mathscr{W}(t), \mathbf{H}^{(\mathbf{l})}_{U_0} \mathbf{H}^{(\mathbf{r})}_{U_0}]$. Then we have
\begin{equation}\label{wstCommGAW}
\small
\begin{split}
  & \mathscr{W}(t)^*[\mathbf{G}, \mathbf{A}_{U(t)}]\mathscr{W}(t)(\Phi)=-\mathscr{W}(t)^*[\mathbf{G}, \mathbf{T}^{(\mathbf{l})}_{U(t)^*U(t)}]\mathscr{W}(t)(\Phi) \in L^2_+(\mathbb{R}; \mathbb{C}^{d \times N})\\
   =& \tfrac{1}{2\pi}\sum_{k=1}^d \sum_{j=1}^M   \lim_{\epsilon \to 0^+}\langle \mathscr{W}(t) \Phi, \chi_{\epsilon} \mathbb{E}_{kj}^{(dM)}U(t)\rangle_{L^2_+(\mathbb{R}; \mathbb{C}^{d \times N})} \mathscr{W}(t)^*\left(\mathbb{E}_{kj}^{(dM)}U(t) \right)
   -i\mathscr{W}(t)^*[\mathbf{G}, \mathbf{H}^{(\mathbf{l})}_{U(t)}\mathbf{H}^{(\mathbf{r})}_{U(t)}]\mathscr{W}(t)  (\Phi)\\
   = &\tfrac{1}{2\pi}\left(\sum_{k=1}^d \sum_{j=1}^M   \lim_{\epsilon \to 0^+} \langle \mathbf{T}^{(\mathbf{l})}_{U_0^*}( \Phi), \chi_{\epsilon} \mathbb{E}_{kj}^{(dM)} \rangle_{L^2_+(\mathbb{R}; \mathbb{C}^{d \times M})}  \mathbb{E}_{kj}^{(dM)} \right)U_0  + i[\mathbf{H}^{(\mathbf{l})}_{U_0} \mathbf{H}^{(\mathbf{r})}_{U_0}, \mathscr{W}(t)^* \mathbf{G}\mathscr{W}(t)](\Phi)\\
   = & i[\mathbf{H}^{(\mathbf{l})}_{U_0} \mathbf{H}^{(\mathbf{r})}_{U_0}, \mathscr{W}(t)^* \mathbf{G}\mathscr{W}(t)](\Phi) + \tfrac{1}{2\pi}\widehat{\Phi U_0^*}(0) U_0 =\left(i[\mathbf{H}^{(\mathbf{l})}_{U_0} \mathbf{H}^{(\mathbf{r})}_{U_0}, \mathscr{W}(t)^* \mathbf{G}\mathscr{W}(t)] +\tfrac{1}{2\pi}\mathfrak{m}_{U_0}^{(\mathbf{lr})}\right)( \Phi)
\end{split}
\end{equation}$\forall \Phi \in  \mathscr{W}(t)^* \left(\mathrm{Dom}(\mathbf{G})^{d\times N} \right)$, by plugging \eqref{HATF0POsitiv}, \eqref{proMlrrl} and \eqref{PsiU0st} into formula \eqref{for[TT,G]}. For any $t \in \mathbb{R}$, set $\mathcal{K}(t):=\mathscr{W}(t)^* \mathbf{G}\mathscr{W}(t)$, then formulas \eqref{WODER} and \eqref{wstCommGAW} imply that
\begin{equation}\label{Kappa(t)}
 \tfrac{\mathrm{d}}{\mathrm{d}t} \mathcal{K}(t) =  \mathcal{K}'(t) = \mathscr{W}(t)^*[\mathbf{G}, \mathbf{A}_{U(t)}]\mathscr{W}(t) = i[\mathbf{H}^{(\mathbf{l})}_{U_0} \mathbf{H}^{(\mathbf{r})}_{U_0}, \mathcal{K} (t)] +\tfrac{1}{2\pi}\mathfrak{m}_{U_0}^{(\mathbf{lr})}.
\end{equation}We integrate \eqref{Kappa(t)} and use the definition \eqref{LlrLrl} in order to deduce that
\begin{equation}\label{ConjActFor}
\mathscr{W}(t)^* \mathbf{G}\mathscr{W}(t) = e^{it \mathbf{H}^{(\mathbf{l})}_{U_0} \mathbf{H}^{(\mathbf{r})}_{U_0}} \left(\mathbf{G} + \mathscr{L}_{U_0}^{(\mathbf{lr})}(t) \right)  e^{-it \mathbf{H}^{(\mathbf{l})}_{U_0} \mathbf{H}^{(\mathbf{r})}_{U_0}}.
\end{equation}Choose $d=M$. Thanks to the conjugation formulas \eqref{W(t)starlinea}, \eqref{WstarUt} and  \eqref{ConjActFor}, for any $z \in \mathbb{C}_+$,  
\begin{equation}\label{kjforMszR}
\small
\begin{split}
& \langle \left( \mathbf{G} -z \right)^{-1}(U(t)), \chi_{\epsilon}\mathbb{E}^{(MN)}_{kj} \rangle_{L^2_+(\mathbb{R}; \mathbb{C}^{M \times N})}  = \langle \left(\mathscr{W}(t)^* \mathbf{G} \mathscr{W}(t) -z \right)^{-1}(\mathscr{W}(t)^*U(t)), \mathscr{W}(t)^*(\chi_{\epsilon}\mathbb{E}^{(MN)}_{kj}) \rangle_{L^2_+}  \\
= & \langle  e^{it \mathbf{H}^{(\mathbf{l})}_{U_0} \mathbf{H}^{(\mathbf{r})}_{U_0}} \left(\mathbf{G} + \mathscr{L}_{U_0}^{(\mathbf{lr})}(t) -z \right)^{-1} e^{-it \mathbf{H}^{(\mathbf{l})}_{U_0} \mathbf{H}^{(\mathbf{r})}_{U_0}} (U_0), e^{it \mathbf{H}^{(\mathbf{l})}_{U_0} \mathbf{H}^{(\mathbf{r})}_{U_0}}\left( \chi_{\epsilon}\mathbb{E}^{(MN)}_{kj}\right) \rangle_{L^2_+(\mathbb{R}; \mathbb{C}^{M \times N})} + \mathbf{r}_{\epsilon}^{(2)}\\
= & \langle    \left(\mathbf{G} + \mathscr{L}_{U_0}^{(\mathbf{lr})}(t) -z \right)^{-1} e^{-it \mathbf{H}^{(\mathbf{l})}_{U_0} \mathbf{H}^{(\mathbf{r})}_{U_0}} (U_0),   \chi_{\epsilon}\mathbb{E}^{(MN)}_{kj}  \rangle_{L^2_+(\mathbb{R}; \mathbb{C}^{M \times N})} + \mathbf{r}_{\epsilon}^{(2)},
\end{split}
\end{equation}for any $1 \leq k \leq M$, $1 \leq j \leq N$, where the term $\mathbf{r}_{\epsilon}^{(2)} \in \mathbb{C}$ is given by
\begin{equation}\label{r2epsi}
\mathbf{r}_{\epsilon}^{(2)} := \langle \mathscr{W}(t)^* \left( \mathbf{G} -z \right)^{-1}(U(t)), \mathscr{W}(t)^*(\chi_{\epsilon}\mathbb{E}^{(MN)}_{kj}) - e^{it \mathbf{H}^{(\mathbf{l})}_{U_0} \mathbf{H}^{(\mathbf{r})}_{U_0}} (\chi_{\epsilon}\mathbb{E}^{(MN)}_{kj}) \rangle_{L^2_+(\mathbb{R}; \mathbb{C}^{M \times N})} \to 0, 
\end{equation}as $\epsilon \to 0^+$, due to Cauchy--Schwarz inequality and \eqref{W(t)starlinea}. Plugging formulas \eqref{kjforMszR} and \eqref{r2epsi} into the inversion formula \eqref{invForR} of the Poisson integral  $\underline{U}(t) = \mathscr{P}[U(t)]$, we obtain that
\begin{equation}\label{lrExpFor}
\begin{split}
2 \pi i \underline{U}(t,z) = & \sum_{k=1}^M \sum_{j=1}^N \lim_{\epsilon \to 0^+} \langle (\mathbf{G}-z)^{-1}(U(t)), \chi_{\epsilon}\mathbb{E}^{(MN)}_{kj}  \rangle_{L^2_+(\mathbb{R}; \mathbb{C}^{M \times N})} \mathbb{E}^{(MN)}_{kj} \\
 = & \sum_{k=1}^M \sum_{j=1}^N \langle    \left(\mathbf{G} + \mathscr{L}_{U_0}^{(\mathbf{lr})}(t) -z \right)^{-1} e^{-it \mathbf{H}^{(\mathbf{l})}_{U_0} \mathbf{H}^{(\mathbf{r})}_{U_0}} (U_0),   \chi_{\epsilon}\mathbb{E}^{(MN)}_{kj}  \rangle_{L^2_+(\mathbb{R}; \mathbb{C}^{M \times N})}\mathbb{E}^{(MN)}_{kj} \\
 = &\mathscr{I}  \left(\left(\mathbf{G} + \mathscr{L}_{U_0}^{(\mathbf{lr})}(t) -z \right)^{-1} e^{-it \mathbf{H}^{(\mathbf{l})}_{U_0} \mathbf{H}^{(\mathbf{r})}_{U_0}} (U_0)\right)\in \mathbb{C}^{M \times N}, \quad \forall z \in \mathbb{C}_+,
\end{split}
\end{equation}by \eqref{HATF0POsitiv}. The other explicit formula can be obtained by the conjugation formulas  \eqref{TF=FT}, \eqref{TconjexpHH} and \eqref{TconjinvG+Lrl-z}. Assume that $U \in C^{\infty }\left(\mathbb{R};  H^s_+(\mathbb{R}; \mathbb{C}^{M\times N}) \right)$ solves   \eqref{MSzego} with $U(0) = U_0  \in  H^s_+(\mathbb{R};  \mathbb{C}^{M\times N})$, then $V:=U^{\mathrm{T}} \in C^{\infty }\left(\mathbb{R};  H^s_+(\mathbb{R}; \mathbb{C}^{N\times M}) \right)$ is also a solution of the matrix Szeg\H{o} equation with initial datum $V(0) = U_0^{\mathrm{T}}= \mathfrak{T} (U_0) \in  H^s_+(\mathbb{R}; \mathbb{C}^{N\times M})$. Then the $(\mathbf{lr})$-explicit formula \eqref{lrExpFor} yields that
\begin{equation*}
\small
\begin{split}
& 2 \pi i\underline{U}(t, z) = 2 \pi i\underline{V}(t,z)^{\mathrm{T}} = \mathfrak{T} \mathscr{I}  \left(\left(\mathbf{G} + \mathscr{L}_{\mathfrak{T} (U_0)}^{(\mathbf{lr})}(t) -z \right)^{-1} e^{-it \mathbf{H}^{(\mathbf{l})}_{\mathfrak{T} (U_0)} \mathbf{H}^{(\mathbf{r})}_{\mathfrak{T} (U_0)}} (U_0^{\mathrm{T}})\right) \in \mathbb{C}^{M \times N} \\
 = & \mathscr{I}  \left(\mathfrak{T}\left(\mathbf{G} + \mathscr{L}_{\mathfrak{T} (U_0)}^{(\mathbf{lr})}(t) -z \right)^{-1} \mathfrak{T}  (\mathfrak{T} e^{-it  \mathbf{H}^{(\mathbf{l})}_{\mathfrak{T} (U_0)} \mathbf{H}^{(\mathbf{r})}_{\mathfrak{T} (U_0)}}\mathfrak{T} ) (U_0 )\right) =\mathscr{I}  \left(\left(\mathbf{G} + \mathscr{L}_{U_0}^{(\mathbf{rl})}(t) -z \right)^{-1} e^{-it \mathbf{H}^{(\mathbf{r})}_{U_0} \mathbf{H}^{(\mathbf{l})}_{U_0}} (U_0)\right).
\end{split}
\end{equation*}
\end{proof}

\end{document}